\documentclass[a4paper]{amsart}
\usepackage{amsfonts,amsmath,amsthm,amssymb,color}
\usepackage[all]{xy}

\usepackage[ colorlinks=true, linkcolor=blue, citecolor=blue]{hyperref}

\numberwithin{equation}{section}

\begin{document}

\newtheorem{theorem}{Theorem}[section]
\newtheorem{lemma}[theorem]{Lemma}
\newtheorem{proposition}[theorem]{Proposition}
\newtheorem{corollary}[theorem]{Corollary}

\theoremstyle{definition}
\newtheorem{definition}[theorem]{Definition}
\newtheorem{example}[theorem]{Example}

\theoremstyle{remark}
\newtheorem{remark}[theorem]{Remark}
\newtheorem*{ack}{Acknowledgments}

\newenvironment{magarray}[1]
{\renewcommand\arraystretch{#1}}
{\renewcommand\arraystretch{1}}

\newcommand{\mapor}[1]{\smash{\mathop{\longrightarrow}\limits^{#1}}}
\newcommand{\mapin}[1]{\smash{\mathop{\hookrightarrow}\limits^{#1}}}
\newcommand{\mapver}[1]{\Big\downarrow
\rlap{$\vcenter{\hbox{$\scriptstyle#1$}}$}}
\newcommand{\liminv}{\smash{\mathop{\lim}\limits_{\leftarrow}\,}}

\newcommand{\Set}{\mathbf{Set}}
\newcommand{\Art}{\mathbf{Art}}
\newcommand{\solose}{\Rightarrow}

\newcommand{\specif}[2]{\left\{#1\,\left|\, #2\right. \,\right\}}

\renewcommand{\bar}{\overline}
\newcommand{\de}{\partial}
\newcommand{\debar}{{\overline{\partial}}}
\newcommand{\per}{\!\cdot\!}
\newcommand{\Oh}{\mathcal{O}}
\newcommand{\sA}{\mathcal{A}}
\newcommand{\sB}{\mathcal{B}}
\newcommand{\sC}{\mathcal{C}}
\newcommand{\sD}{\mathcal{D}}
\newcommand{\sE}{\mathcal{E}}
\newcommand{\sF}{\mathcal{F}}\newcommand{\sG}{\mathcal{G}}
\newcommand{\sH}{\mathcal{H}}
\newcommand{\sI}{\mathcal{I}}
\newcommand{\sJ}{\mathcal{J}}
\newcommand{\sL}{\mathcal{L}}
\newcommand{\sM}{\mathcal{M}}
\newcommand{\sP}{\mathcal{P}}
\newcommand{\sU}{\mathcal{U}}
\newcommand{\sV}{\mathcal{V}}
\newcommand{\sX}{\mathcal{X}}
\newcommand{\sY}{\mathcal{Y}}
\newcommand{\sN}{\mathcal{N}}
\newcommand{\sZ}{\mathcal{Z}}

\newcommand{\Aut}{\operatorname{Aut}}
\newcommand{\Mor}{\operatorname{Mor}}
\newcommand{\Def}{\operatorname{Def}}
\newcommand{\Hom}{\operatorname{Hom}}
\newcommand{\Hilb}{\operatorname{Hilb}}
\newcommand{\HOM}{\operatorname{\mathcal H}\!\!om}
\newcommand{\DER}{\operatorname{\mathcal D}\!er}
\newcommand{\Spec}{\operatorname{Spec}}
\newcommand{\Der}{\operatorname{Der}}
\newcommand{\End}{{\operatorname{End}}}
\newcommand{\END}{\operatorname{\mathcal E}\!\!nd}
\newcommand{\Image}{\operatorname{Im}}
\newcommand{\coker}{\operatorname{coker}}
\newcommand{\tot}{\operatorname{tot}}
\newcommand{\Diff}{\operatorname{Diff}}
\newcommand{\ten}{\otimes}
\newcommand{\mA}{\mathfrak{m}_{A}}

\renewcommand{\Hat}[1]{\widehat{#1}}
\newcommand{\dual}{^{\vee}}
\newcommand{\desude}[2]{\dfrac{\de #1}{\de #2}}
\newcommand{\sK}{\mathcal{K}}
\newcommand{\A}{\mathbb{A}}
\newcommand{\N}{\mathbb{N}}
\newcommand{\R}{\mathbb{R}}
\newcommand{\Z}{\mathbb{Z}}
\renewcommand{\H}{\mathbb{H}}
\renewcommand{\L}{\mathbb{L}}
\newcommand{\proj}{\mathbb{P}}
\newcommand{\K}{\mathbb{K}\,}
\newcommand\C{\mathbb{C}}
\newcommand\T{\mathbb{T}}
\newcommand\Del{\operatorname{Del}}
\newcommand\Tot{\operatorname{Tot}}
\newcommand\Grpd{\mbox{\bf Grpd}}
\newcommand\rif{~\ref}

\newcommand\vr{``}
\newcommand{\rh}{\rightarrow}
\newcommand{\contr}{{\mspace{1mu}\lrcorner\mspace{1.5mu}}}

\newcommand{\bi}{\boldsymbol{i}}
\newcommand{\bl}{\boldsymbol{l}}

\newcommand{\MC}{\operatorname{MC}}
\newcommand{\Coder}{\operatorname{Coder}}
\newcommand{\TW}{\operatorname{TW}}
\newcommand{\id}{\operatorname{id}}
\newcommand{\ad}{\operatorname{ad}}
\newcommand{\cone}{\operatorname{C}}
\newcommand{\cylinder}{\operatorname{Cyl}}

\title{Nonabelian higher derived brackets}
\author{Ruggero Bandiera}
\address{\newline
Universit\`a degli studi di Roma La Sapienza,\hfill\newline
Dipartimento di Matematica \lq\lq Guido
Castelnuovo\rq\rq,\hfill\newline
P.le Aldo Moro 5,
I-00185 Roma, Italy.}
\email{bandiera@mat.uniroma1.it}

\date{September 9, 2013.}

\begin{abstract} Let $M$ be a graded Lie algebra, together with graded Lie subalgebras $L$ and $A$ such that as a graded space $M$ is the direct sum of $L$ and $A$, and $A$ is abelian. Let $D$ be a degree one derivation of $M$ squaring to zero and sending $L$ into itself, then Voronov's construction of higher derived brackets associates to $D$ a $L_\infty$ structure on $A[-1]$. It is known, and it follows from the results of this paper, that the resulting $L_\infty$ algebra is a weak model for the homotopy fiber of the inclusion of differential graded Lie algebras $i:(L,D,[\cdot,\cdot])\rh (M,D,[\cdot,\cdot])$. We prove this fact using homotopical transfer of $L_\infty$ structures, in this way we also extend Voronov's construction when the assumption $A$ abelian is dropped: the resulting formulas involve Bernoulli numbers. In the last section we consider some example and some further application.
\end{abstract}

\maketitle

\section{Introduction}

Let $i:(L,D,[\cdot,\cdot])\rh(M,D,[\cdot,\cdot])$ be the inclusion of a differential graded Lie subalgebra, recall that its homotopy fiber is the differential graded Lie algebra (dgla)
\[K_i=\{(l,m(t,dt))\in L\times M[t,dt]\,\,\,\mbox{s.t.}\,\,\,m(t,dt)_{|t=0}=0,\,m(t,dt)_{|t=1}=l\},\]
where $M[t,dt]$ is the dgla of polynomial forms on the line with coefficients in $M$. Let $A\subset M$ be a complement to $L$ in $M$ and $P:M\rh A$ the projection with kernel $L$, then $(A[-1],-PD)$ is a homotopy retract of $K_i$: via homotopy transfer of $L_\infty$ structures there is an induced $L_\infty$ structure on $A[-1]$, together with a homotopy fiber sequence $A[-1]\rh L\xrightarrow{i}M$ of $L_\infty$ algebras. In this paper we find explicit formulas under the additional assumption that $A\subset M$ is a graded Lie subalgebra of $M$, then the $L_\infty$ structure on $A[-1]$ is given (after d\'ecalage) by the family of degree one symmetric brackets $\Phi(D)_i:A^{\odot i}\rh A$, for $i\geq1$,
\[ \Phi(D)_i(a_1\odot\cdots\odot a_i)=\sum_{\sigma\in S_i}\varepsilon(\sigma)\sum_{k=1}^{i}\frac{B_{i-k}}{k!(i-k)!}\overbrace{[\cdots[}^{i-k}P([\cdots[Da_{\sigma(1)},a_{\sigma(2)}]\cdots,a_{\sigma(k)}]),a_{\sigma(k+1)}]\cdots,a_{\sigma(i)}] \]
where $S_i$ is the symmetric group, $\varepsilon(\sigma)=\varepsilon(\sigma;a_1,\ldots,a_i)$ is the Koszul sign and the $B_j$ are the Bernoulli numbers. The $L_\infty$ morphism $A[-1]\rh L$, modeling the projection $p_L:K_i\rh L$, is given (after d\'ecalage) in Taylor coefficients by
\begin{equation}\label{morfismo}A^{\odot i}\rh L[1]:a_1\odot\cdots\odot a_i\rh\frac{1}{i!}\sum_{\sigma\in S_i}\varepsilon(\sigma)P^\bot[\cdots[Da_{\sigma(1)},a_{\sigma(2)}]\cdots,a_{\sigma(i)}],\qquad\mbox{$i\geq1$,}\end{equation}
where $P^\bot:=\id_M-P:M\rh L$. Our first main result is that with these definitions
\begin{theorem}\label{ThoeremHDBvshomotopyfiber} If $A$ is a graded Lie subalgebra of $M$ then $(A[-1],\Phi(D)_1=PD,\ldots,\Phi(D)_i,\ldots)$ is a $L_\infty$ algebra and a weak model for the homotopy fiber $K_i$ in the homotopy category of $L_\infty$ algebras. Moreover, $A[-1]\rh L\xrightarrow{i}M$ is a homotopy fiber sequence of $L_\infty$ algebras (more precisely, a weak model for the sequence $K_i\xrightarrow{p_L} L\xrightarrow{i}M$).\end{theorem}
When $A\subset M$ is an abelian Lie subalgebra, then $\Phi(D)_i(a_1\odot\cdots\odot a_i)=P[\cdots[Da_1,a_2]\cdots,a_i]$, which are the higher derived brackets on $A$ associated to $D$ introduced by Th. Voronov in \cite{voronov2}: in this case the first part of the above Theorem is proved in \cite{voronov2}, Section 4.

Following \cite{voronov2}, Section 4, we also construct a $L_\infty$ structure on $M\times A[-1]$ and a diagram
\[\xymatrix{&M\times A[-1]\ar[rd]\ar@<-2pt>[dd]&\\A[-1]\ar[ru]\ar[rd]&&M\\&L\ar[ru]\ar@<-2pt>[uu]&}\]
such that: the lower sequence is the one from the previous theorem and the upper sequence is a $L_\infty$ extension (in the sense of \cite{Lazarev}, Definition 3.1) of base $M$ and fibre $A[-1]$, the vertical arrows are quasi inverses $L_\infty$ quasi isomorphisms, and finally the right and left triangles are separately commutative, and the right one can be regarded as the homotopy replacement of the inclusion $i$ by a fibration. It might be interesting to point out that we deduce Theorem~\ref{ThoeremHDBvshomotopyfiber} as a particular case of a more general result: namely, if $j:(N,D,[\cdot,\cdot])\rh(M,D,[\cdot,\cdot])$ is the inclusion of another sub dgla, then the $L_\infty$ structure on $M\times A[-1]$ restricts to one on $N\times A[-1]$, which gives a weak model for the homotopy fiber product
$N\times^h_M L$ (Definition \ref{def:homfibprod}) along the inclusions $j$ and $i$.

Let $\Phi(D)$ be the coderivation on the reduced symmetric coalgebra $\overline{SA}$ over $A$ associated to the brackets $\Phi(D)_i$ (under the isomorphism $\Coder(\overline{SA})\cong\Hom(\overline{SA},A)=\prod_{i\geq1}\Hom(A^{\odot i},A)$ given by corestriction), then by definition of $L_\infty$ algebra $\Phi(D)$ induces a dg coalgebra structure on $\overline{SA}$, thus also on the (non reduced) symmetric coalgebra $SA$. According to the classification of $L_\infty$ extensions from \cite{ChLaz2,methazambon,Lazarev}, the $L_\infty$ extension $A[-1]\rh M\times A[-1]\rh M$ is classified by a $L_\infty$ morphism $\Phi:(M,D,[\cdot,\cdot])\rh(\Coder(SA),[\Phi(D),\cdot],[\cdot,\cdot])$, where the bracket on $\Coder(SA)$ is the usual (Nijenhuis-Richardson) bracket of coderivations: we show that $\Phi$ is a strict morphism of dglas, moreover, when $A\subset M$ is a graded Lie subalgebra we obtain explicit formulas for the Taylor coefficients $\Phi(m)_i:A^{\odot i}\rh A$, and these are again given by higher derived brackets
\[ \Phi(m)_i(a_1\odot\cdots\odot a_i)=\sum_{\sigma\in S_i}\varepsilon(\sigma)\sum_{k=0}^{i}\frac{B_{i-k}}{k!(i-k)!}\overbrace{[\cdots[}^{i-k}P([\cdots[m,a_{\sigma(1)}]
\cdots,a_{\sigma(k)}]),a_{\sigma(k+1)}]\cdots,a_{\sigma(i)}]\]
for $i\geq1$, with moreover the 0-th bracket $\Phi(m)_0:A^{\odot0}=\K\rh A:1\rh Pm$. In fact, when $A\subset M$ is an abelian Lie subalgebra these are the higher derived brackets on $A$ associated to $m$ introduced by Voronov in \cite{voronov}.

This homotopical construction of higher derived brackets also implies the formal properties proved by Voronov in \cite{voronov,voronov2} by a direct computation. 
\begin{theorem}\label{TheoremHDB} Let $M$ be a graded Lie algebra, $L$ and $A$ graded Lie subalgebras such that $M=L\oplus A$ as graded spaces: for $m\in M$ (resp.: $D\in\Der(M/L)$, cf. the subsection on notations) we define the higher derived brackets $\Phi(m)_i:A^{\odot i}\rh A$, $i\geq0$, (resp.: $\Phi(D)_i:A^{\odot i}\rh A$, $i\geq1$) on $A$ associated to $m$ (resp.: $D$) by the above formulas, cf. Definition \ref{def.HDB}. We denote by $\Phi(m)\in\Coder(SA)$ (resp.: $\Phi(D)\in\Coder(\overline{SA})$ ) the corresponding coderivation. For every $D,D_k\in\Der(M/L)$, $m,m_k\in M$, $k=1,2$, the following identities hold:
\begin{equation}\label{HDB.id.1} [\Phi(m_1),\Phi(m_2)] = \Phi([m_1,m_2]), \end{equation}
\begin{equation}\label{HDB.id.2} [\Phi(D_1),\Phi(D_2)] = \Phi([D_1,D_2]), \end{equation}
\begin{equation}\label{HDB.id.3} [\Phi(D),\Phi(m)] = \Phi(Dm), \end{equation}
where the bracket in the left hand side is the usual (Nijenhuis-Richardson) bracket of coderivations.\end{theorem}

The first few brackets are given by
\begin{eqnarray} && \Phi(D)_1(a)=PDa,\qquad\Phi(m)_1(a)=P[m,a]-\frac{1}{2}[Pm,a], \nonumber \\
&& \Phi(D)_2(a_1\odot a_2)=\sum_{\sigma\in S_2}\varepsilon(\sigma)\left(\frac{1}{2}P[Da_{\sigma(1)},a_{\sigma(2)}]-\frac{1}{2}[PDa_{\sigma(1)},a_{\sigma(2)}]\right), \nonumber \end{eqnarray}
\[ \Phi(m)_2(a_1\odot a_2)=\sum_{\sigma\in S_2}\varepsilon(\sigma)\left(\frac{1}{2}P[[m,a_{\sigma(1)}],a_{\sigma(2)}]-\frac{1}{2}[P[m,a_{\sigma(1)}],a_{\sigma(2)}] +\frac{1}{12}[[Pm,a_{\sigma(1)}],a_{\sigma(2)}]\right). \]

We remark that in general the brackets associated to $m$ are not the same as those associated to the inner derivation $[m,\cdot]$ (even when $[m,\cdot]\in\Der(M/L)$, that is, $m$ is in the normalizer of $L$ in $M$) unless $m\in L$, so the two constructions shouldn't be confused. As for the appearance of Bernoulli numbers this follows from the computations made by Fiorenza and Manetti in \cite{FMcone}, of which we make extensive use, cf. also the related computations in \cite{Bering,Getzler,Merkulov}. 

Both Theorem \ref{ThoeremHDBvshomotopyfiber} and Theorem \ref{TheoremHDB} will be proved in Section \ref{sectionHDB}: the proofs will depend on the construction, by Fiorenza, Manetti \cite{FMcone} and Iacono \cite{iacono0}, of  the $L_\infty$ mapping cocone and the $L_\infty$ mapping cocylinder of a morphism of dglas, which will be reviewed in Section \ref{section3}, the classification of $L_\infty$ extensions, which will be reviewed in Section \ref{section4}, and finally homotopical transfer of $L_\infty$ structures, which will be reviewed in Section \ref{section2}. In Remark \ref{rem.generalizedbrackets} we discuss how both theorems generalize when we further remove the hypothesis that $A\subset M$ is a graded Lie subalgebra (then $A$ is just supposed to be a complement of the graded Lie subalgebra $L$ in $M$).

A good source of motivation for the study of derived brackets and higher derived brackets comes from physics and differential geometry, cf. \cite{KosmannSch,voronov,voronov2,Bering}. Another list of examples from deformation theory \cite{FrZam1,FrZam2,Sch,CS08} can be better understood in light of our Theorem \ref{ThoeremHDBvshomotopyfiber}: in fact homotopy fibers naturally occur when considering semitrivial deformation problems (cf. \cite{Man,FMcone}). Consider for instance the case of deformations of a coisotropic submanifold of a Poisson manifold, where the usual approach \cite{Sch,CS08,FrZam1}, using higher derived brackets, can be paralleled with the more recent one in \cite{BM}, using homotopy fibers, cf. Example~\ref{ex.Poisson}.

In Section \ref{sec-examples} we consider some example and some further application, these include: a theorem of Chuang and Lazarev \cite{ChLaz}, stating that every $L_\infty$ algebra $(V[-1],Q)$ (that is, by definition, $Q\in\Coder^1(\overline{SV})$ and $[Q,Q]=0$) is weakly equivalent to the homotopy fiber of the inclusion of dglas
\[ i:(\Coder(\overline{SV}),[Q,\cdot],[\cdot,\cdot])\rh(\Coder(SV),[Q,\cdot],[\cdot,\cdot]),\]
which also implies a simple necessary and sufficient condition for a $L_\infty$ algebra to be homotopy abelian (a non trivial application of this last result is given in \cite{FoKBopLA}); the construction, via higher derived brackets, of a $L_\infty$ structure on the suspension of the negatively graded part of any dgla, already due to Getzler \cite{Getzler}, generalizing the well known case of quantum type dglas; the construction of a hierarchy of higher brackets (conjecturally the same as the one introduced by Bering~\cite{Bering}) associated to an operator over a unital associative graded algebra $A$, reducing to the usual higher Koszul brackets \cite{koszul} when $A$ is commutative; homotopy abelianity of the $L_\infty$ algebras associated to commutative $BV_\infty$ algebras satisfying the degeneration property (the author is grateful to Marco Manetti for pointing out and carefully explaining to him this example), which has been proved with different methods by Braun and Lazarev \cite{BrLaz}; the extension of the results of the author and Manetti \cite{BM} to the study of coisotropic deformations in the differentiable setting.

\begin{ack} It is a pleasure to thank my Ph.D. advisor Marco Manetti for guiding me through the writing of this paper, as well as Domenico Fiorenza for numerous and useful discussions. I finally thank Alessandro D'Andrea and the referee for some useful suggestions.\end{ack}

\subsection{Notations and conventions}

We work over a field $\K$ of characteristic zero, graded means $\Z$-graded, differentials raise the degree by one. If $V=\oplus_{i\in\Z}V^i$ is a graded space and $k\in\Z$ then $V[k]$ is the graded space defined by $V[k]^i=V^{i+k}$, we use a particular notation for the desuspension $s^{-1}V:=V[1]$ of $V$. We denote by $S(V)=\oplus_{i\geq0}V^{\odot i}$ (resp.: $\overline{S}(V)=\oplus_{i\geq1}V^{\odot i}$) the (resp.: reduced) symmetric coalgebra over $V$ (where $V^{\odot n}$ is the $n$-th symmetric tensor power of $V$, i.e., the space of \emph{coinvariants} of $V^{\otimes n}$ under the natural action of the symmetric group $S_n$, with the usual Koszul rule for twisting signs, $V^{\odot 0}=\K$) and by $\odot$ the symmetric tensor product; sometimes we simplify the notations to $SV$ and $\overline{SV}$. For the rest of our conventions we refer the reader to \cite{FMcone}, Section 1.

For a graded Lie algebra $M$ and a graded Lie subalgebra $L\subset M$ we denote by $\Der(M)$ the graded Lie algebra of derivations of $M$ and by $\Der(M/L)\subset\Der(M)$ the graded Lie subalgebra of derivations $D$ such that $D(L)\subset L$.

We will work in the category of $L_\infty[1]$ algebras and $L_\infty[1]$ morphisms between them: this is isomorphic to the usual category of $L_\infty$ algebras from \cite{LaSt,Getzler04} via the so called d\'ecalage isomorphisms, cf. for instance \cite{FMcone} (we recall that in this way $L_\infty$ algebra structures on a graded space $V$ correspond to $L_\infty[1]$ algebra structures on the desuspension $s^{-1}V$).

\section{Review of $L_{\infty}[1]$ algebras}\label{section2}

Let $p:SV\xrightarrow{} V^{\odot1}=V $ (resp.: $p:\overline{SV}\xrightarrow{} V$) denote the natural projection, then corestriction induces isomorphisms of graded spaces from the the graded Lie algebras of coderivations
\[ \Coder(SV)\xrightarrow{\cong} \Hom(SV,V)=\prod_{i\geq0}\Hom(V^{\odot i},V):Q\xrightarrow{\qquad}pQ  \]
\[ \Coder(\overline{SV})\xrightarrow{\cong} \Hom(\overline{SV},V)=\prod_{i\geq1}\Hom(V^{\odot i},V):Q\xrightarrow{\qquad}pQ  \]
If $Q\in\Coder(SV)$ and $pQ=q=(q_0,\ldots,q_i,\ldots)$ we call $q_i:V^{\odot i}\rh V$ the $i$-th Taylor coefficient of $Q$, similarly if $Q\in\Coder(\overline{SV})$. The inverse of the first isomorphism sends $(q_0,\ldots,q_i,\ldots)$ to the coderivation given by $Q(1) = q_0(1)\in V\subset SV$ and
\[ Q(v_1\odot\cdots\odot v_i)=q_0(1)\odot v_1\odot\cdots\odot v_i + \sum_{k=1}^{i}\sum_{\sigma\in S(k,i-k)}\varepsilon(\sigma) q_k(v_{\sigma(1)}\odot\cdots\odot v_{\sigma(k)})\odot v_{\sigma(k+1)}\odot\cdots\odot v_{\sigma(i)} \]
where $S(i,j)$ is the set of $(i,j)$-unshuffles, i.e., permutations $\sigma\in S_{i+j}$ such that $\sigma(1)<\cdots<\sigma(i)$ and $\sigma(i+1)<\cdots<\sigma(i+j)$, and $\varepsilon(\sigma)=\varepsilon(\sigma;v_1,\ldots,v_i)$ is the Koszul sign. The inverse of the second isomorphism is given by the above formula minus the term $q_0(1)\odot v_1\odot\cdots\odot v_i$. Coderivations such that $q_i=0$ for $i\neq1$ are called linear.

\begin{remark}\label{rem.exsequence1} There is a natural embedding $i:\Coder(\overline{SV})\rh\Coder(SV)$, given in Taylor coefficients by $(q_1,\ldots,q_i,\ldots)\xrightarrow{} (0,q_1,\ldots,q_i,\ldots)$, which identifies $\Coder(\overline{SV})$ with the graded Lie subalgebra of coderivations $Q\in\Coder(SV)$ such that $Q(1)=0$. It fits into an exact sequence of graded spaces
\begin{equation}\label{exsequence1} 0\xrightarrow{} \Coder(\overline{SV})\xrightarrow{\,\,i\,\,} \Coder(SV)\xrightarrow{\,\,e\,\,} V\xrightarrow{} 0, \end{equation}
where $e$ is the evaluation morphism $Q\xrightarrow{} Q(1)=q_0(1)$.\end{remark}

The usual commutator Lie bracket on $\Coder(SV)$ is induced from a right pre-Lie product, which we call the  Nijenhuis-Richardson product and denote by $\bullet$: namely, $Q\bullet R$ is the coderivation which corestricts to $pQR$. In Taylor coefficients
\[ (Q\bullet R)_i(v_1\odot\cdots\odot v_i) = \sum_{k=0}^{i}\sum_{\sigma\in S(k,i-k)}\varepsilon(\sigma) q_{i-k+1}(r_k(v_{\sigma(1)}\odot\cdots\odot v_{\sigma(k)})\odot v_{\sigma(k+1)}\odot\cdots\odot v_{\sigma(i)}) \]
with the convention $r_0(\varnothing):=r_0(1)$. The bracket on $\Coder(SV)$ and the  Nijenhuis-Richardson product are related by $[Q,R]=(Q\bullet R)-(-1)^{|Q||R|}(R\bullet Q)$.

\begin{remark}\label{remark2.1} Given $v\in V$, let $\sigma_v\in\Coder(SV)$ be the coderivation defined by
\[ \sigma_v(1)=v ,\qquad \sigma_v(v_1\odot\cdots\odot v_i)=v\odot v_1\odot\cdots\odot v_i. \]
In Taylor coefficients $p\sigma_v=(j_v,0,\ldots,0,\ldots)$, where $j_v:V^{\odot 0}\rh V:1\rh v$. Then the section $\sigma:V\rh\Coder(SV):v\rh \sigma_v$ splits the exact sequence \eqref{exsequence1}, and the image is an abelian Lie subalgebra of $\Coder(SV)$.
More in general, for every $v\in V$ and $Q\in\Coder(SA)$ it is plain to see that in Taylor cofficients $[Q,\sigma_v]_{i}(v_1\odot\cdots\odot v_i)=q_{i+1}(v\odot v_1\odot\cdots\odot v_i)$,  for every $i\geq0$.
\end{remark}

Given graded spaces $V$ and $W$, then corestriction $F\rh pF=(f_1,\ldots,f_i,\ldots)$ induces a bijective correspondence between the set of morphisms of coalgebras $F:\overline{SV}\xrightarrow{}\overline{SW}$ and the set $\Hom^0(\overline{SV}, W)\cong\prod_{i\geq1}\Hom^0(V^{\odot i}, W)$ of morphisms of graded spaces from $\overline{SV}$ to $W$. We can reconstruct $F$ from its corestriction via the formula
\[ F^{k}_{i}(v_1\odot\cdots\odot v_i)=\frac{1}{k!}\sum_{j_{1}+\cdots+j_{k}=i}\sum_{\sigma\in S(j_{1},\ldots,j_{k})}\varepsilon(\sigma)f_{j_{1}}(v_{\sigma(1)}\odot\cdots)\odot\cdots\odot f_{j_{k}}(\cdots\odot v_{\sigma(i)}), \]
where $F^k_i$ is the composition $V^{\odot i}\xrightarrow{}\overline{SV}\xrightarrow{F}\overline{SW}\xrightarrow{} W^{\odot k}$, and $F^k_i=0$ if $k>i$. Recall (\cite{KoSo}) that $F$ is an isomorphism (resp.: monomorphism, epimorphism) of coalgebras if and only if its linear part $f_1$ is an isomorphism (resp.: monomorphism, epimorphism) of graded spaces. A morphism $F$ such that $f_i=0$ for $i\geq 2$ is called linear.

\begin{definition} A $L_\infty[1]$ algebra structure $(V,Q)=(V,q_1,\ldots,q_i,\ldots)$ on a graded space $V$ is the datum of a differential graded (dg) coalgebra structure on $\overline{SV}$, that is a degree one coderivation $Q\in\Coder^1(\overline{SV})$, $pQ=(q_1,\ldots,q_i,\ldots)$, such that $Q^2=Q\bullet Q=\frac{1}{2}[Q,Q]=0$. In particular $(q_1)^2=0$: the dg space $(V,q_1)$ is called the tangent complex of the $L_\infty[1]$ algebra $(V,Q)$. \end{definition}

\begin{remark} Given a $L_\infty[1]$ structure $Q$ on $V$, the exact sequence \eqref{exsequence1} enriches to an exact sequence of dg spaces
\begin{equation}\label{exsequence2} 0\xrightarrow{} (\Coder(\overline{SV}),[Q,\cdot])\xrightarrow{\,\,i\,\,} (\Coder(SV),[Q,\cdot])\xrightarrow{\,\,e\,\,} (V,q_1)\xrightarrow{} 0 \end{equation}\end{remark}
\begin{definition} Given $L_\infty[1]$ algebras $(V,Q)$ and $(W,R)$ a $L_\infty[1]$ morphism between them is a morphism of dg coalgebras $F:(\overline{SV},Q)\xrightarrow{} (\overline{SW},R)$, i.e., a morphism of coalgebras such that $FQ-RF=0$. We denote by $\mathbf{L}_\infty[1]$ the category of $L_\infty[1]$ algebras and $L_\infty[1]$ morphisms between them.

A linear $L_\infty[1]$ morphism is also called strict. A morphism of $L_\infty[1]$ algebras is a weak equivalence if its linear part is a quasi isomorphism between the tangent complexes.\end{definition}
\begin{definition} A $L_\infty[1]$ algebra $(V,Q)$ is abelian if $Q$ is a linear coderivation (that is $q_i=0$ for $i\geq2$). It is homotopy abelian if it is weakly equivalent to an abelian $L_\infty[1]$ algebra.\end{definition}

\begin{remark}\label{rem-mixedterms} To give the Taylor coefficients of a $L_\infty[1]$ structure on the direct product $V\times W$ of graded spaces $V$ and $W$, or of a $L_\infty[1]$ morphism with domain $V\times W$,  it is convenient to decompose the symmetric powers of $V\times W$ into types $(V\times W)^{\odot i}\cong\bigoplus_{j=0}^i V^{\odot i-j}\otimes W^{\odot j}$. For any graded space $X$ we see that $\Hom(\overline{S}(V\times W), X)=\prod_{i\geq1}\Hom((V\times W)^{\odot i},X)=\prod_{j+k\geq1}\Hom(V^{\odot j}\otimes W^{\odot k}, X)$ .  \end{remark}

\begin{example}\label{ex:quillenconstr} Given a differential graded Lie algebra (dgla) $(L,d,[\cdot,\cdot])$, it is defined a $L_{\infty}[1]$ structure $Q$ on $s^{-1}L=L[1]$ by $q_1(s^{-1}l)=-s^{-1}dl$, $q_2(s^{-1}l_1\odot s^{-1}l_2)=(-1)^{|l_1|}s^{-1}[l_1,l_2]$ and $q_i=0$ for $i\geq 3$: we denote this $L_\infty[1]$ algebra $(s^{-1}L,Q)$ by $\Sigma^{-1}L$, or sometimes $\Sigma^{-1}(L,d,[\cdot,\cdot])$. This defines a functor $\Sigma^{-1}:\mathbf{DGLA}\rh\mathbf{L}_\infty[1]$, by sending a morphism $f:L\rh M$ of dglas to the strict $L_\infty[1]$ morphism $\Sigma^{-1}(f):\Sigma^{-1}L\rh\Sigma^{-1}M$ given by $\Sigma^{-1}(f)_1(s^{-1}l)=s^{-1}f(l)$ and $\Sigma^{-1}(f)_i=0$ for $i\geq2$.
\end{example}

Recall \cite{Getzler04} that the curvature of a $L_\infty[1]$ algebra $(V,Q)=(V,q_1,\ldots,q_n,\ldots)$ is the function
\[ \mathcal{R}:V^0\rh V^1:x\rh \sum_{i\geq1}\frac{1}{i!}q_i(x^{\odot i})\]
Usually, to make sense of the above infinite sum one adds some assumption on $(V,Q)$, for instance that it is nilpotent or that it is complete with respect to some filtration compatible with the $L_\infty[1]$ structure. We will adopt a less restrictive point of view in Remark \ref{rem.twisting}, for the moment we will just ignore any convergence issue.
\begin{definition} $\MC(V):=\{x\in V^0\,\,\,\mbox{s.t.}\,\,\,\mathcal{R}(x)=0\}$ is called the Maurer-Cartan set of the $L_\infty[1]$ algebra $(V,Q)$, its elements the Maurer-Cartan elements of $(V,Q)$.\end{definition}
It is well known \cite{Getzler04} that given a $L_\infty[1]$ algebra $(V,Q)$ and a Maurer-Cartan element $x\in\MC(V)$, we can twist $Q$ by $x$ to obtain a new $L_\infty[1]$ structure $Q_x$ on $V$, given in Taylor coefficients $pQ_x=(q_{x,1},\ldots,q_{x,i},\ldots)$ by
\begin{equation}\label{twistingstructure} q_{x,i}(v_1\odot\cdots\odot v_i)=\sum_{j\geq0}\frac{1}{j!}q_{i+j}(x^{\odot j}\odot v_1\odot\cdots\odot v_i).\end{equation}
We are going to need a relative version of this result, which can be found, for instance, in \cite{yeku}. Let $F:(V,Q)\rh(W,R)$ be a $L_\infty[1]$ morphism of $L_\infty[1]$ algebras, $pF=(f_1,\ldots,f_i,\ldots)$. Let $x\in\MC(V)$, there is an induced Maurer-Cartan element $\MC(F)(x)\in\MC(W)$: this is given by $\MC(F)(x)=\sum_{i\geq1}\frac{1}{i!}f_i(x^{\odot i})$. We can twist the $L_\infty[1]$ structures $Q$ and $R$ by $x$ and $\MC(F)(x)$ respectively, as in \eqref{twistingstructure}. Finally, we can twist $F$ by $x$ to obtain a new morphism of coalgebras $F_x:\overline{SV}\rh\overline{SW}$, given in Taylor coefficients $pF_x=(f_{x,1},\ldots,f_{x,i},\ldots)$ by
\begin{equation} f_{x,i}(v_1\odot\cdots\odot v_i)=\sum_{j\geq0}\frac{1}{j!}f_{i+j}(x^{\odot j}\odot v_1\odot\cdots\odot v_i).\end{equation}
A proof of the following theorem can be found in \cite{yeku}.
\begin{theorem}\label{twistingmorphism} $F_x:(V,Q_x)\rh(W,R_{\MC(F)(x)})$ is a $L_\infty[1]$ morphism of $L_\infty[1]$ algebras.\end{theorem}
\begin{remark}\label{rem.twisting} As already remarked some hypothesis is needed to ensure convergence of the above infinite sums. For our purposes it is more convenient to impose such a hypothesis on $x$ rather than on $V$. Suppose that for a particular choice of $x\in V^0$ all of the above infinite sums become finite (so that we can make sense of $Q_x$, $R_{\MC(F)(x)}$ and $F_x$) and $\mathcal{R}(x)=0$: then Theorem \ref{twistingmorphism} holds. \end{remark}
We close this section by recalling the theorem on homotopical transfer of $L_\infty[1]$ structures: this says that $L_\infty[1]$ structures (unlike, for instance, dgla structures) can be transferred along homotopy retractions. It is a major result in the theory of $L_\infty[1]$ algebras, implying for instance the existence of the minimal model \cite{KoSo}. The version we give here is taken from \cite{FMcone}, Theorem 4.1, for a nice proof the reader is referred to the arXiv version of the paper.

\begin{definition}\label{def.contractiondata} Given a pair of dg spaces $(V,q_1)$ and $(W,r_1)$ we call homotopy retraction data from $V$ to $W$ the data of a pair of dg morphisms $\pi:V\xrightarrow{} W$, $f_{1}:W\xrightarrow{} V$ and a contracting homotopy $K\in\Hom^{-1}(V,V)$ such that
\[ \pi f_1=\id_W \]
\[ Kq_1 + q_1K = f_1\pi - \id_V \]
\end{definition}

\begin{theorem}\label{homtranstheorem} Let $(V,Q)$ be a $L_\infty[1]$ algebra, $(W,r_1)$ a dg space and $\pi$, $f_1$, $K$ homotopy retraction data from $(V,q_1)$ to $(W,r_1)$ as in the previous definition. Let $q=pQ$, $q_+=q-q_1$, then there exists an unique morphism of coalgebras $F:\overline{SW}\xrightarrow{}\overline{SV}$ such that $pF=f_1+Kq_+F$. Moreover, if $R\in\Coder^1(\overline{SW})$ is the coderivation wich corestricts to $pR=r_1 + \pi q_+F$, then $R$ is a $L_\infty[1]$ structure on $W$ and $F:(W,R)\xrightarrow{}(V,Q)$ is a $L_\infty[1]$ morphism.
\end{theorem}

\begin{remark}\label{rem-inversetoF} Notice that homotopical transfer of structure produces a weakly equivalent $L_\infty[1]$ algebra, with $F:(W,R)\rh(V,Q)$ an explicit weak equivalence. Using a more refined version of the above theorem \cite{berglund} it is also possible to construct a $L_\infty[1]$ right inverse $G:(V,Q)\rh(W,R)$ to $F$, with linear part $g_1=\pi$.\end{remark}

\section{Mapping cocone and mapping cocylinder of a dgla morphism}\label{section3}

Given a dgla $(L,d_L,[\cdot,\cdot])$ and a commutative differential graded algebra (dga) $(A,d_A,\cdot)$ there is a natural dgla structure on the tensor product $A\otimes L$, with differential $d_A\otimes\id_L+\id_A\otimes d_L$ and bracket $[a_1\otimes l_1,a_2\otimes l_2]=(-1)^{|l_1||a_2|}a_1a_2\otimes[l_1,l_2]$. When $A=\K[t,dt]$ is the commutative dga of polynomial forms on the line (that is the graded commutative $\K$-algebra freely generated by $t$ in degree zero and $dt$ in degree one, and with differential $d(t)=dt$ and $d(dt)=0$) then the corresponding dgla will be denoted by $L[t,dt]$. Every $s\in\K$ induces a a quasi isomorphism given by the evaluation $e_s:L[t,dt]\xrightarrow{t=s,dt=0} L$, which is left inverse to the natural inclusion $L\xrightarrow{} L[t,dt]:l\xrightarrow{}1\otimes l$. There are degree minus one operators $\int_0^t:L[t,dt]\xrightarrow{} L[t]$ and $\int_0^1:L[t,dt]\xrightarrow{} L$ induced by formal integration in $dt$.

\begin{definition}\label{def:homfibprod} Given a pair of morphisms of dglas $f:L\xrightarrow{} M$ and $g: N\xrightarrow{} M$ the homotopy fiber product of $L$ and $N$ along $f$ and $g$ is the dgla
\[ L\times_M^h N=\{ (l,n,m(t,dt))\in L\times N\times M[t,dt] \mbox{\,\,s.t.\,\,} e_0(m(t,dt))=f(l),\, e_1(m(t,dt))=g(n)\} \]
Given a morphism $f:L\xrightarrow{} M$ of dglas we denote by
\[ K_f = 0 \times_M^h L\qquad\mbox{the homotopy fiber product along $0$ and $f$.}\]
We call $K_f$ the homotopy fiber of $f$.
\end{definition}

\begin{definition}\label{def:homfibseq} A short sequence of $L_\infty[1]$ algebras and $L_\infty[1]$ morphisms which is weakly equivalent to one of the form $\Sigma^{-1}K_f\xrightarrow{\Sigma^{-1}(p_L)}\Sigma^{-1}L\xrightarrow{\Sigma^{-1}(f)}\Sigma^{-1}M$ (cf. Example \ref{ex:quillenconstr}), for some morphism $f:L\rh M$ of dglas, is called a homotopy fiber sequence of $L_\infty[1]$ algebras.\end{definition}

The next theorem, due to M. Manetti, gives a sufficient condition for a homotopy fiber to be homotopy abelian. The reader is referred to \cite{iacMan1,iacMan2} for a proof and some nice applications in deformation theory.

\begin{theorem}\label{th:Manetti} Let $i:L\rh M$ be the inclusion of a sub dgla: if $H(i):H(L)\rh H(M)$ is injective, then the homotopy fiber $K_i$ is homotopy abelian.\end{theorem}

We denote by $\{B_i\}_{i\in\mathbb{N}}$ the sequence of Bernoulli numbers, i.e., the sequence of (rational) numbers defined by the power series expansion $\frac{t}{e^t-1}=\sum_{i\geq0}\frac{B_i}{i!}t^i=1-\frac{1}{2}t+\frac{1}{2!}\frac{1}{6}t^{2}+\frac{1}{4!}(-\frac{1}{30})t^{4}+\cdots$. This is sometimes called the first sequence of Bernoulli numbers, occasionally we will use the notation $B_i(0):=B_i$ to distinguish it from the second sequence of Bernoulli numbers, defined by $B_{i}(1):=(-1)^iB_i$. Recall that $B_{2i+1}(0)=B_{2i+1}(1)=0$ for $i\geq 1$.

We recall a construction, due to D. Fiorenza and M. Manetti \cite{FMcone} in the case of a homotopy fiber, and which has been further elaborated by D. Iacono \cite{iacono0}, of a $L_\infty[1]$ model of the homotopy fiber product of a pair of morphisms of dglas. This will play a central role in the computations in Section \ref{sectionHDB}.

Let $g:N\rh M$ and $f:L\rh M$ be a pair of morphisms of dglas. Let $N\times^h_M L$ be the homotopy fiber product along $g$ and $f$ and let $\Sigma^{-1}(N\times^h_M L))=(s^{-1}N\times s^{-1}L\times s^{-1}M[t,dt],Q)$ be the corresponding $L_\infty[1]$ algebra, as in Example \ref{ex:quillenconstr}. Let $(s^{-1}\cone_{g,f}, r_1)$ be the dg space
\[ s^{-1}\cone_{g,f}=s^{-1}N\times s^{-1}L\times M,\qquad r_1(s^{-1}n,s^{-1}l,m)=\left(-s^{-1}d_Nn,-s^{-1}d_Ll,d_Mm + g(n)-f(l)\right)\]
We can define homotopy retraction data, as in Definition \ref{def.contractiondata}, from $(s^{-1}N\times s^{-1}L\times s^{-1}M[t,dt],q_1)$ to $(s^{-1}\cone_{g,f},r_1)$ as follows
\begin{equation}\label{eq:retrcone1}f_1(s^{-1}n,s^{-1}l, m) = \left(s^{-1}n,s^{-1}l,s^{-1} \left((1-t)\cdot g(n) + t\cdot f(l) +dt\cdot m\right)\right),\end{equation}
\begin{equation}\label{eq:retrcone2}\pi(s^{-1}n,s^{-1}l,s^{-1}m(t,dt))=\left(s^{-1}n,s^{-1}l,\int^1_0m(t,dt)\right),\end{equation}
\begin{equation}\label{eq:retrcone3}K(s^{-1}n,s^{-1}l,s^{-1}m(t,dt))=\left(0,0,s^{-1}\left(\int_0^t m(t,dt)-t\cdot\int_0^1 m(t,dt)\right)\right).\end{equation}

\begin{theorem}\label{mappingcocylinder} (\cite{iacono0}) The $L_\infty[1]$ structure $R$ on $(s^{-1}\cone_{g,f},r_1)$, induced via homotopy transfer according to Theorem \ref{homtranstheorem}, is given in Taylor coefficients $pR=(r_1,\ldots,r_i,\ldots)$ by (cf. Remark~\ref{rem-mixedterms})
\[ r_2(s^{-1}n_1\odot s^{-1}n_2)=(-1)^{|n_1|}s^{-1}[n_1,n_2],\qquad r_2(s^{-1}l_1\odot s^{-1}l_2)=(-1)^{|l_1|}s^{-1}[l_1,l_2], \]
\[ r_{i+1}(s^{-1}n\otimes m_1\odot\cdots\odot m_i)=\frac{B_i(1)}{i!}\sum_{\sigma\in S_i}\varepsilon(\sigma)[\cdots[g(n),m_{\sigma(1)}]\cdots,m_{\sigma(i)}],\qquad i\geq1, \]
\[ r_{i+1}(s^{-1}l\otimes m_1\odot\cdots\odot m_i)=-\frac{B_i(0)}{i!}\sum_{\sigma\in S_i}\varepsilon(\sigma)[\cdots[f(l),m_{\sigma(1)}]\cdots,m_{\sigma(i)}],\qquad i\geq1, \]
and $R=0$ otherwise.
\end{theorem}

We distinguish two particular cases of the above construction.
\begin{definition}\label{def.mappingcoc-etc} Let $f:L\rh M$ be a morphism of dglas, then:
\begin{enumerate} \item the desuspended mapping cocylinder of $f$ is the $L_\infty[1]$ algebra $s^{-1}\cylinder_f=s^{-1}\cone_{\id_M,f}$, with the $L_\infty[1]$ structure from Theorem \ref{mappingcocylinder} (in the case $N=M$ and $g=\id_M:M\rh M$);
\item the desuspended mapping cocone of $f$ is the $L_\infty[1]$ algebra $s^{-1}\cone_f=s^{-1}\cone_{0,f}$, with the $L_\infty[1]$ structure from Theorem~\ref{mappingcocylinder} (in the case $N=0$).
\end{enumerate}
\end{definition}

\section{$L_\infty[1]$ extensions}\label{section4}

In this section we review the classification of $L_\infty[1]$ extensions from \cite{ChLaz2,methazambon,Lazarev}. Recall that a $L_\infty[1]$ ideal of a $L_\infty[1]$ algebra $(V,q_1,\ldots,q_i,\ldots)$ is a subspace $I\subset V$ such that $q_{i}(I\otimes V^{\odot i-1})\subset I$ for every $i\geq1$: then the quotient $V/I$ inherits a $L_\infty[1]$ structure. The following definition rigidify the one in \cite{Lazarev}, Definition 3.1.

\begin{definition}\label{def:Looextensions} Let $(V,Q)$ and $(W,R)$ be $L_\infty[1]$ algebras. Let $\Theta$ be a $L_\infty[1]$ structure on $V\times W$ such that $(W,R)\xrightarrow{0\times\id_W}(V\times W,\Theta)$ is the inclusion of a $L_\infty[1]$ ideal, and the induced $L_\infty[1]$ structure on $V\times W/\{0\}\times W\cong V$ is $Q$. We call $(W,R)\rh (V\times W,\Theta) \rh (V,Q)$ a $L_\infty[1]$ extension of base $(V,Q)$ and fibre $(W,R)$.\end{definition}

A proof of the following result can be found in \cite{ChLaz2}, Corollary 3.6, or \cite{methazambon}, Proposition 6.1.

\begin{proposition}\label{th:Looextensions} There is a bijective correspondence between the set of $L_\infty[1]$ extensions of base $(V,Q)$ and fiber $(W,R)$ and the set of $L_\infty[1]$ morphisms $(V,Q)\rh\Sigma^{-1}(\Coder(SW),[R,\cdot],[\cdot,\cdot])$ (cf. Example \ref{ex:quillenconstr}). The corresponding $L_\infty[1]$ morphism $pF=(f_1,\ldots,f_i,\ldots)$ and the $L_\infty[1]$ structure $p\Theta=(\theta_1,\ldots,\theta_i,\ldots)$ on $V\times W$ determine each other via the formulas (cf. Remark \ref{rem-mixedterms})
\[\theta_i(v_1\odot\cdots\odot v_i)=(q_i(v_1\odot\cdots\odot v_{i}),sf_i(v_1\odot\cdots\odot v_i)_0(1)),\]
\[\theta_{j}(w_1\odot\cdots\odot w_j)=(0,r_j(w_1\odot\cdots\odot w_j))\qquad\mbox{and}\]
\[\theta_{i+j}(v_1\odot\cdots\odot v_i\otimes w_1\odot\cdots\odot w_j)=(0,sf_i(v_1\odot\cdots\odot v_i)_j(w_1\odot\cdots\odot w_j))\qquad\mbox{for $i,j\geq1$}.\]\end{proposition}
\begin{remark} By $sf_i(\cdots)$ we mean the composition $V^{\odot i}\xrightarrow{f_i} s^{-1}\Coder(SW)\xrightarrow{s} \Coder(SW)$, where $s$ is the shift map, while $sf_i(\cdots)_j$ means that we take the $j$-th Taylor coefficient. \end{remark}

The next lemma will be needed, in combination with Proposition \ref{th:Looextensions}, in the proof of Proposition~\ref{prop.HDB}.

\begin{lemma}\label{lemma:Psi} Let $i:L\rh M$ be the inclusion of a graded Lie subalgebra (seen as a morphism of dglas with trivial differentials) and let $(s^{-1}\cylinder_i,R)$ be the desuspended mapping cocylinder of $i$, as in Definition~\ref{def.mappingcoc-etc}. Then $\Psi:(\Der(M/L),0,[\cdot,\cdot])\rh(\Coder(\overline{S}(s^{-1}\cylinder_i)),[R,\cdot],[\cdot,\cdot])$, given by \[\Psi(D)_1(s^{-1}l,s^{-1}m,n)=((-1)^{|D|}s^{-1}Dl,(-1)^{|D|}s^{-1}Dm,Dn)\qquad\mbox{and $\Psi(D)_i=0$ for $i\geq2$,}\]
is a morphism of dglas.\end{lemma}
\begin{proof} The only non trivial thing to prove is $[R,\Psi(D)]=0$ for all $D\in\Der(M/L)$: this is equivalent to $[r_i,\Psi(D)_1]=0$ for every $i\geq1$, which can be checked directly, since $D$ is a derivation and the formula for $r_i$ involves nested brackets. We sketch a different approach: let $H_i$ be the homotopy fiber product $M\times_M^h L$ along $\id_M$ and $i$, seen as morphisms of dglas with trivial differentials. We have an obvious inclusion  $\Der(M/L)\rh\Der(H_i)$, so that we can form the semidirect product $\Der(M/L)\rtimes H_i$. The homotopy retraction data from $s^{-1}H_i$ to $s^{-1}\cylinder_i$ in \eqref{eq:retrcone1}-\eqref{eq:retrcone3} induces homotopy retraction data from $\Sigma^{-1}(\Der(M/L)\rtimes H_i)$ to $s^{-1}\Der(M/L)\times s^{-1}\cylinder_i$. The transferred $L_\infty[1]$ structure can be computed explicitly, along the lines of \cite{FMcone} and \cite{iacono0}, and one verifies that this is a $L_\infty[1]$ extension of base $\Sigma^{-1}(\Der(M/L),0,[\cdot,\cdot])$ and fibre $(s^{-1}\cylinder_i,R)$: finally, the classifying $L_\infty[1]$ morphism, as in Proposition \ref{th:Looextensions}, is exactly $\Sigma^{-1}(\Psi)$, thus $\Psi$ is a morphism of dglas. We will use an analogous argument in the proof of Theorem \ref{TheoremHDB}.\end{proof}

\section{Nonabelian higher derived brackets}\label{sectionHDB}

Let $M$ be a graded Lie algebra, together with a projection $P:M\rh M$ (i.e., $P^2=P$) such that both $L=\operatorname{Ker}\,P$ and $A=\operatorname{Im}\,P$ are graded Lie subalgebras of $M$; let $P^\bot=\id_M-P$. We denote by $\Der(M/L)\subset\Der(M)$ the Lie subalgebra of derivations $D\in\Der(M)$ such that $D(L)\subset L$, which is equivalent to $PDP^\bot=0$, that is,
\begin{equation}\label{PDP=PD} PDP = PD. \end{equation}
We are going to extend to this setting Voronov's constructions of higher derived brackets \cite{voronov,voronov2}.

\begin{definition}\label{def.HDB} For $m\in M$ the higher derived brackets on $A$ associated to $m$ are the multilinear graded symmetric maps $\Phi(m)_i:A^{\odot i}\xrightarrow{} A$, $i\geq0$, defined by
\[ \Phi(m)_i(a_1\odot\cdots\odot a_i)=\sum_{\sigma\in S_i}\varepsilon(\sigma)\sum_{k=0}^{i}\frac{B_{i-k}}{k!(i-k)!}\overbrace{[\cdots[}^{i-k}P([\cdots[m,a_{\sigma(1)}]\cdots,a_{\sigma(k)}]),a_{\sigma(k+1)}]\cdots,a_{\sigma(i)}] \]
for $i\geq1$, and $\Phi(m)_0(1)=Pm$. We denote by $(\Phi(m)_0,\ldots,\Phi(m)_i,\ldots)=\Phi(m)\in\Coder(SA)$ the corresponding coderivation.

For $D\in\Der(M/L)$ the higher derived brackets on $A$ associated to $D$ are the multilinear graded symmetric maps $\Phi(D)_i:A^{\odot i}\rh A$, $i\geq1$, defined by
\[ \Phi(D)_i(a_1\odot\cdots\odot a_i)=\sum_{\sigma\in S_i}\varepsilon(\sigma)\sum_{k=1}^{i}\frac{B_{i-k}}{k!(i-k)!}\overbrace{[\cdots[}^{i-k}P([\cdots[Da_{\sigma(1)},a_{\sigma(2)}]\cdots, a_{\sigma(k)}]),a_{\sigma(k+1)}]\cdots,a_{\sigma(i)}] \]
We denote by $(\Phi(D)_1,\ldots,\Phi(D)_i,\ldots)=\Phi(D)\in\Coder(\overline{SA})$ the corresponding coderivation.
\end{definition}

\begin{remark}\label{remark5.2} If $l\in L$, then the inner derivation $[l,\cdot]$ is in $\Der(M/L)$, and in this case the two constructions coincide: we can (and will) identify $\Phi([l,\cdot])=\Phi(l)$. However, if $[m,\cdot]\in\Der(M/L)$, that is, $m$ is in the normalizer of $L$ in $M$, but $m\not\in L$, then in general $\Phi([m,\cdot])\neq\Phi(m)$, as the two differ by the terms involving $Pm$.\end{remark}

If $A\subset M$ is an abelian Lie subalgebra only the $k=i$ term in the summation for the brackets remains. Moreover, it can be proved with a simple induction, as in \cite{voronov2}, that for any derivation $D\in\Der(M)$ the maps $A^{\otimes i}\rh M:a_1\otimes\cdots\otimes a_i\rh[\cdots[D a_1,a_2]\cdots,a_i]$ are graded symmetric. From this the following proposition follows straightforwardly.

\begin{proposition}\label{abeliancase} If $A$ is an abelian Lie subalgebra of $M$, then for $m\in M$ the brackets in Definition~\ref{def.HDB} reduce to $\Phi(m)_i(a_1\odot\cdots\odot a_i)=P[\cdots[m,a_1]\cdots,a_i]$, $i\geq1$, $\Phi(m)_0(1)=Pm$,  while they reduce to $\Phi(D)_i(a_1\odot\cdots\odot a_i)=P[\cdots[Da_1,a_2]\cdots,a_i]$, $i\geq1$, for $D\in\Der(M/L)$: i.e., they are the same as those introduced by Th. Voronov in \cite{voronov,voronov2}.
\end{proposition}

The aim of this section is to prove Theorem \ref{ThoeremHDBvshomotopyfiber} and Theorem \ref{TheoremHDB} from the introduction. Both will follow from the next proposition, where we imitate a construction in \cite{voronov2}, Theorem 2.

\begin{proposition}\label{prop.HDB} There is a $L_\infty[1]$ structure $pR=(r_1,\ldots,r_n,\ldots)$ on $s^{-1}\Der(M/L)\times s^{-1}M\times A$, given in Taylor coefficients by (cf. Remark \ref{rem-mixedterms})
\begin{eqnarray} && r_1(s^{-1}D,s^{-1}m,a)= (0,0,Pm) \nonumber \\
&& r_2(s^{-1}m_1\odot s^{-1}m_2) = (-1)^{|m_1|}s^{-1}[m_1,m_2] \nonumber \\
&& r_2(s^{-1}D_1\odot s^{-1}D_2) = (-1)^{|D_1|}s^{-1}[D_1,D_2] \nonumber \\
&& r_2(s^{-1}D\otimes s^{-1}m) = (-1)^{|D|}s^{-1}Dm \nonumber \\
&& r_{i+1}(s^{-1}D\otimes a_1\odot\cdots\odot a_i) = \Phi(D)_{i}(a_1\odot\cdots\odot a_i) \nonumber \\
&& r_{i+1}(s^{-1}m\otimes a_1\odot\cdots \odot a_i) = \Phi(m)_i(a_1\odot\cdots\odot a_i) \nonumber \end{eqnarray}
for $i\geq1$, and $R=0$ otherwise.\end{proposition}

\begin{proof} Let $s^{-1}\cylinder_i$ be the desuspended mapping cocylinder for the inclusion $i$, with the $L_\infty[1]$ structure defined in Theorem \ref{mappingcocylinder} (in the case $d_L=d_M=0$). By combining Proposition~\ref{th:Looextensions} and Lemma~\ref{lemma:Psi} (cf. also the proof of the latter), it is defined a $L_\infty[1]$ structure $Q$ on the space $s^{-1}\Der(M/L)\times s^{-1}\cylinder_i=s^{-1}\Der(M/L)\times(s^{-1}M\times s^{-1}L\times M)$: explicitly formulas can be derived from those in \ref{mappingcocylinder}, \ref{th:Looextensions} and \ref{lemma:Psi}.

We consider the following homotopy retraction data from $(s^{-1}\Der(M/L)\times s^{-1}\cylinder_i,q_1)$ to $(s^{-1}\Der(M/L)\times s^{-1}M\times A,r_1)$, with $r_1$ as in the claim of the proposition:
\begin{equation}\label{retrdata1} \pi:s^{-1}\Der(M/L)\times s^{-1}\cylinder_i\xrightarrow{} s^{-1}\Der(M/L)\times s^{-1}M\times A : \end{equation}\[:\left(s^{-1}D,\,(s^{-1}m,\,s^{-1}l,\, n)\,\right) \xrightarrow{\qquad}\left(s^{-1}D,\, s^{-1}m,\,Pn\right)\]
\begin{equation}\label{retrdata2} f_1: s^{-1}\Der(M/L)\times s^{-1}M\times A \xrightarrow{} s^{-1}\Der(M/L)\times s^{-1}\cylinder_i:\end{equation}\[:(s^{-1}D,s^{-1}m,a)\xrightarrow{\qquad} \left(s^{-1}D,(s^{-1}m,\,s^{-1}P^\bot m,,\,a)\right) \]
\begin{equation}\label{retrdata3} K: s^{-1}\Der(M/L)\times s^{-1}\cylinder_i\xrightarrow{} s^{-1}\Der(M/L)\times s^{-1}\cylinder_i:\end{equation} \[:\left(s^{-1}D,(s^{-1}m,\,s^{-1}l,\,n)\right)\xrightarrow{\qquad} \left(0,(0,\,s^{-1}P^\bot n,\,0)\right) \]
We leave to the reader to check that this is in fact homotopy retraction data as in Definition~\ref{def.contractiondata}. We are going to prove that $R$ is the $L_\infty[1]$ structure on $s^{-1}\Der(M/L)\times s^{-1}M\times A$ induced from $Q$ via homotopical transfer of $L_\infty[1]$ structures, as in Theorem \ref{homtranstheorem}.

Let $F:\overline{S}(s^{-1}\Der(M/L)\times s^{-1}M\times A)\rh\overline{S}(s^{-1}\Der(M/L)\times s^{-1}\cylinder_i)$ be the morphism of coalgebras in the claim of  Theorem \ref{homtranstheorem}: recall that its Taylor coefficients are defined inductively so that the linear one is $f_1$, and for $i\geq2$ we have $f_i=\sum_{j=2}^{2} Kq_j F^j_i$, where $F^j_i$ (cf. Section \ref{section2}) is given by the formula
\[ F^{j}_{i}(\cdots)=\frac{1}{j!}\sum_{k_{1}+\cdots+k_{j}=i}\sum_{\sigma\in S(k_{1},\ldots,k_{j})}\varepsilon(\sigma)f_{k_{1}}(\cdots)\odot\cdots\odot f_{k_{j}}(\cdots). \]
As $K$ factors through the inclusion $s^{-1}L\rh s^{-1}\cylinder_i\rh  s^{-1}\Der(M/L)\times s^{-1}\cylinder_i$, so does $f_i$ for every $i\geq2$. By looking at the explicit formulas for $q_j$, $j\geq2$, we see that
\[ F^j_i(\cdots) = \frac{1}{(j-1)!}\sum_{\sigma\in S(i-j+1,1,\ldots,1)}\varepsilon(\sigma)f_{i-j+1}(\cdots)\odot f_1(\cdots)\odot\cdots\odot f_1(\cdots) + \left\{\mbox{terms in $\operatorname{Ker}(Kq_j)$}\right\}. \]
An inductive analysis of the several possibilities implies that $f_{i+1}$, for $i\geq1$, vanishes everywhere but on mixed terms of type (cf. Remark \ref{rem-mixedterms}) $s^{-1}D\otimes a_{1}\odot\cdots\odot a_i$ and $s^{-1}m\otimes a_1\odot\cdots\odot a_i$.

\begin{remark}\label{rem-nonLie} The reader will notice that up to now everything works fine without the assumption that $A\subset M$ is a graded Lie subalgebra (this, however, will be essential in the sequent computation).\end{remark}

We show that $f_{i+1}$, for $i\geq1$, is explicitly given by
\begin{equation}\label{morfismoF1} f_{i+1}(s^{-1}D\otimes a_1\odot\cdots\odot a_i)=\left(0,\left(0,\,s^{-1}\frac{1}{i!}\sum_{\sigma\in S_i}\varepsilon(\sigma)P^\bot [\cdots[Da_{\sigma(1)},a_{\sigma(2)}]\cdots,a_{\sigma(i)}],\,0 \right)\,\right),\end{equation}
\begin{equation}\label{morfismoF2} f_{i+1}(s^{-1}m\otimes a_1\odot\cdots\odot a_i)=\left(0,\left(0,\,s^{-1}\frac{1}{i!}\sum_{\sigma\in S_i}\varepsilon(\sigma)P^\bot [\cdots[m,a_{\sigma(1)}]\cdots,a_{\sigma(i)}],\,0 \right)\,\right), \end{equation}
and $f_{i+1}=0$ otherwise.

The reader will check directly (recall \eqref{PDP=PD}) that $f_2=Kq_2 f_1^{\odot 2}$, so we suppose $i\geq2$. We have to prove $f_{i+1}=\sum_{j=2}^{i+1}Kq_j F^j_{i+1}$. To simplify the computation we notice that we are only interested in keeping track of $p_Mq_j F^{j}_{i+1}$, where we denote by $p_M$ the natural projection $s^{-1}\cylinder_i\rh M$: in fact $Kq_jF^{j}_{i+1}=\left(0,\left(0,s^{-1}P^\bot(p_Mq_jF^j_{i+1}),0\right)\right)$. The considerations which preceded Remark \ref{rem-nonLie} imply that for $2\leq j\leq i$
\[ p_Mq_j F^j_{i+1}(s^{-1}D\otimes a_1\odot\cdots\odot a_i) = \]
\[ = \sum_{\sigma\in S(i-j+1,j-1)}\varepsilon(\sigma)p_Mq_j\left(f_{i-j+2}(s^{-1}D\otimes a_{\sigma(1)}\odot\cdots\odot a_{\sigma(i-j+1)})\ten a_{\sigma(i-j+2)}\odot\cdots\odot a_{\sigma(i)}\right) = \]
\[  = -\frac{B_{j-1}}{(j-1)!(i-j+1)!}\sum_{\sigma\in S_i}\varepsilon(\sigma)\overbrace{[\cdots[}^{j-1}P^\bot([\cdots[Da_{\sigma(1)},a_{\sigma(2)}]\cdots]),a_{\sigma(i-j+2)}]\cdots,a_{\sigma(i)}]  \]
In the first identity we used symmetry of $f_{i-j+2}$ and $q_j$ to deduce that $\frac{1}{(j-1)!}\sum_{\sigma\in S(i-j+1,1,\ldots,1)}\cdots= \frac{1}{(j-1)!(i-j+1)!}\sum_{\sigma\in S_i}\cdots=\sum_{\sigma\in S(i-j+1,j-1)}\cdots$, where the suspension points must be filled as in the second term of the identity. In the same way, for $2\leq j\leq i$
\[ p_Mq_j F^j_{i+1}(s^{-1}m\otimes a_1\odot\cdots\odot a_i) = \]
\[  = -\frac{B_{j-1}}{(j-1)!(i-j+1)!}\sum_{\sigma\in S_i}\varepsilon(\sigma)\overbrace{[\cdots[}^{j-1}P^\bot([\cdots[m,a_{\sigma(1)}]\cdots]),a_{\sigma(i-j+2)}]\cdots,a_{\sigma(i)}]  \]
The remaining terms to consider are
\[ p_Mq_{i+1}F^{i+1}_{i+1}(s^{-1}D\otimes a_1\odot\cdots\odot a_i) =  0,  \]
for $i\geq2$, and
\[ p_Mq_{i+1}F^{i+1}_{i+1}(s^{-1}m\otimes a_1\odot\cdots\odot a_i)=p_Mq_{i+1}(s^{-1}m\otimes a_1\odot\cdots\odot a_i)+p_Mq_{i+1}(s^{-1}P^\bot m\otimes a_1\odot\cdots\odot a_i) = \]
\[ = \frac{B_{i}}{i!}\sum_{\sigma\in S_i}\varepsilon(\sigma)[\cdots[m-P^\bot m,a_{\sigma(1)}]\cdots,a_{\sigma(i)}] = \frac{B_{i}}{i!}\sum_{\sigma\in S_i}\varepsilon(\sigma)[\cdots[Pm,a_{\sigma(1)}]\cdots,a_{\sigma(i)}]\]

Finally, after a change of variable $k=i-j+1$, we see that
\[ \sum_{j=2}^{i+1}p_M q_j F^j_{i+1}(s^{-1}D\otimes a_1\odot\cdots\odot a_i) = \left( -\sum_{k=1}^{i-1}\frac{B_{i-k}}{k!(i-k)!} \right)\sum_{\sigma\in S_i}\varepsilon(\sigma)[\cdots[Da_{\sigma(1)},a_{\sigma(2)}]\cdots,a_{\sigma(i)}] + \]
\[ + \sum_{k=1}^{i-1} \frac{B_{i-k}}{k!(i-k)!}\sum_{\sigma\in S_i}\varepsilon(\sigma)\overbrace{[\cdots[}^{i-k}P([\cdots[Da_{\sigma(1)},a_{\sigma(2)}]\cdots]),a_{\sigma(k+1)}]\cdots,a_{\sigma(i)}]  \]
We use the well known identity $\sum_{k=0}^{i-1} B_k\left(\begin{array}{c} i \\ k \end{array}\right)=0$, for $i\geq2$, in order to conclude
\begin{equation}\label{p_Mq_+Fcaso1} \sum_{j=2}^{i+1}p_M q_j F^j_{i+1}(s^{-1}D\otimes a_1\odot\cdots\odot a_i) = \frac{1}{i!}\sum_{\sigma\in S_i}\varepsilon(\sigma)[\cdots[Da_{\sigma(1)},a_{\sigma(2)}]\cdots,a_{\sigma(i)}] + \end{equation}
\[ + \sum_{k=1}^{i-1}\frac{B_{i-k}}{k!(i-k)!}\sum_{\sigma\in S_i}\varepsilon(\sigma)\overbrace{[\cdots[}^{i-k}P([\cdots[Da_{\sigma(1)},a_{\sigma(2)}]\cdots]),a_{\sigma(k+1)}]\cdots,a_{\sigma(i)}] \]
In a similar way
\begin{equation}\label{p_Mq_+Fcaso2} \sum_{j=2}^{i+1}p_M q_j F^j_{i+1}(s^{-1}m\otimes a_1\odot\cdots\odot a_i) =\frac{1}{i!}\sum_{\sigma\in S_i}\varepsilon(\sigma)[\cdots[m,a_{\sigma(1)}]\cdots,a_{\sigma(i)}] + \end{equation}
\[ + \sum_{k=0}^{i-1} \frac{B_{i-k}}{k!(i-k)!}\sum_{\sigma\in S_i}\varepsilon(\sigma)\overbrace{[\cdots[}^{i-k}P([\cdots[m,a_{\sigma(1)}]\cdots]),a_{\sigma(k+1)}]\cdots,a_{\sigma(i)}] \]
Since $A$ is $[\cdot,\cdot]$-closed, in both identities the bottom line lies in $A=\operatorname{Ker}\,P^\bot$: this implies, as desired, $f_{i+1} = \sum_{j=2}^{i+1}Kq_j F^j_{i+1}$.

In order to complete the proof we have to show $r_{i+1}=\sum_{j=2}^{i+1}\pi q_j F^{j}_{i+1}$ for $i\geq1$, where $r_{i+1}$ is defined as in the claim of the proposition: the reader will check directly that $r_2=\pi q_2 f_1^{\odot 2}$, so we suppose $i\geq2$. We already computed $p_Mq_jF^j_{i+1}$, on the other hand  $p_{s^{-1}\Der(M/L)}q_j F^j_{i+1}=p_{s^{-1}M}q_j F^j_{i+1}=0$ for $i\geq2$ and $2\leq j\leq i+1$: let's prove for instance that for $p_{s^{-1}M}q_j F^j_{i+1}=0$, for $j>2$ this follows from $p_{s^{-1}M}q_j=0$, in the remaining case from the fact that $p_{s^{-1}M}:(s^{-1}\Der(M/L)\times s^{-1}\cylinder_i,R)\rh\Sigma^{-1}(M,0,[\cdot,\cdot])$ is a strict morphism of $L_\infty[1]$ algebras and from $p_{s^{-1}M}^{\odot 2}F^2_{i+1}=0$ for $i\geq2$, since $p_{s^{-1}M}f_j=0$ for $j\geq2$. Similarly we see that $p_{s^{-1}\Der(M/L)}q_j F^j_{i+1}=0$ for $i\geq2$ and $2\leq j\leq i+1$: thus, for $i\geq2$, $\pi q_jF^j_{i+1}=\left(0,0,P(p_Mq_jF^j_{i+1})\right)$. A comparison between Definition \ref{def.HDB} and the previous equations \eqref{p_Mq_+Fcaso1}-\eqref{p_Mq_+Fcaso2} shows $r_{i+1}=\sum_{j=2}^{i+1}\pi q_j F^{j}_{i+1}$ and proves the proposition.
\end{proof}

Theorem \ref{TheoremHDB} follows directly from the previous proposition and Proposition \ref{th:Looextensions}.

\begin{proof}(of Theorem \ref{TheoremHDB}) Let $\Der(M/L)\rtimes M$ be the obvious semidirect product, Theorem \ref{TheoremHDB} is equivalent to say that $\Phi:\Der(M/L)\rtimes M\rh\Coder(SA):(D,m)\rh\Phi(D)+\Phi(m)$ is a morphism of graded Lie algebras. The $L_\infty[1]$ structure from the previous proposition fits into a $L_\infty[1]$ extension of base $\Sigma^{-1}(\Der(M/L)\rtimes M,0,[\cdot,\cdot])$ (cf. Example \ref{ex:quillenconstr}) and fibre $(A,0)$, which is classified by a $L_\infty[1]$ morphism $\Sigma^{-1}(\Der(M/L)\rtimes M,0,[\cdot,\cdot])\rh\Sigma^{-1}(\Coder(SA),0,[\cdot,\cdot])$, according to Proposition~\ref{th:Looextensions}. Finally, the explicit form of the correspondence in Proposition \ref{th:Looextensions} implies that the classifying morphism is exactly $\Sigma^{-1}(\Phi)$, thus $\Phi$ is a morphism of graded Lie algebras.\end{proof}

Theorem \ref{ThoeremHDBvshomotopyfiber} will follow as a particular case of a more general result. We consider the $L_\infty[1]$ structure $Q$ on $s^{-1}\Der(M/L)\times s^{-1}\cylinder_i$ as in the first paragraph of the proof of Proposition~\ref{prop.HDB}: then in the same proof we constructed a $L_\infty[1]$ morphism
\[ F:(s^{-1}\Der(M/L)\times s^{-1}M\times A,R)\rh (s^{-1}\Der(M/L)\times s^{-1}\cylinder_i,Q), \]
in fact a weak equivalence, cf. Equations \eqref{morfismoF1} and \eqref{morfismoF2}. At this point we remark that if $D\in\Der^1(M/L)$ satisfies $[D,D]=0$ then $(s^{-1}D,0,0)\in(s^{-1}\Der(M/L)\times s^{-1}M\times A)^0$ satisfies the assumptions in Remark \ref{rem.twisting}, so it makes sense to twist everything by $D$ to get a new $L_\infty[1]$ morphism
\[ F_D:(s^{-1}\Der(M/L)\times s^{-1}M\times A,R_D)\rh (s^{-1}\Der(M/L)\times s^{-1}\cylinder_i,Q_D). \]
Let $j:(N,D,[\cdot,\cdot])\rh(M,D,[\cdot,\cdot])$ be the inclusion of a sub dgla, then $R_D$ restricts to a $L_\infty[1]$ structure on $s^{-1}N\times A$, still denoted by $R_D$, explicitly given by
\begin{equation}\label{eq:modelfibredproduct1}r_{D,1}(s^{-1}n,a)=\left(-s^{-1}Dn,P(Da+n)\right),\qquad r_{D,2}(s^{-1}n_1\odot s^{-1}n_2)=(-1)^{|n_1|}s^{-1}[n_1,n_2],\end{equation}
\begin{equation}\label{eq:modelfibredproduct2}r_{D,i+1}(s^{-1}n\otimes a_1\odot\cdots\odot a_i)=\Phi(n)_i(a_1\odot\cdots\odot a_1),\qquad i\geq1,\end{equation}
\begin{equation}\label{eq:modelfibredproduct3}r_{D,i}(a_1\odot\cdots\odot a_i)=\Phi(D)_i(a_1\odot\cdots\odot a_i),\qquad i\geq2,\end{equation}
and $R_D=0$ otherwise. Similarly $Q_D$ restricts on $s^{-1}\cone_{j,i}=s^{-1}N\times s^{-1}L\times M\subset s^{-1}\cylinder_i$ to a $L_\infty[1]$ structure, still denoted by $Q_D$, which is exactly Iacono's model (Theorem \ref{mappingcocylinder}) for the homotopy fiber product $N\times^h_M L$ along $j:(N,D,[\cdot,\cdot])\rh(M,D,[\cdot,\cdot])$ and $i:(L,D,[\cdot,\cdot])\rh(M,D,[\cdot,\cdot])$. The $L_\infty[1]$ morphism $F_D$ restricts to a $L_\infty[1]$ morphism $F_D:(s^{-1}N\times A,R_D)\rh (s^{-1}\cone_{j,i},Q_D)$, explicilty given by
\begin{equation}\label{eq:F_D1}f_{D,1}(s^{-1}n,a)=\left(s^{-1}n,s^{-1}P^\bot(n+Da),a\right),\end{equation}
\begin{equation}\label{eq:F_D2}f_{D,i+1}(s^{-1}n\otimes a_1\odot\cdots\odot a_i)=\left(0,s^{-1}\frac{1}{i!}\sum_{\sigma\in S_i}\varepsilon(\sigma)P^\bot[\cdots[n,a_{\sigma(1)}]\cdots,a_{\sigma(i)}],0\right),\qquad i\geq1,\end{equation}
\begin{equation}\label{eq:F_D3}f_{D,i}(a_1\odot\cdots\odot a_i)=\left(0,s^{-1}\frac{1}{i!}\sum_{\sigma\in S_i}\varepsilon(\sigma)P^\bot[\cdots[Da_{\sigma(1)},a_{\sigma(2)}]\cdots,a_{\sigma(i)}],0\right),\qquad i\geq2,\end{equation}
and $F_D=0$ otherwise. Finally, we notice that $F_D$ is a weak equivalence of $L_\infty[1]$ algebras: in fact $\pi:(s^{-1}\cone_{j,i},q_{1,D})\rh(s^{-1}N\times A,r_{1,D}):(s^{-1}n,s^{-1}l,m)\rh(s^{-1}n,Pm)$ is a dg right inverse to $f_{D,1}$, and $K(s^{-1}n,s^{-1}l,m)=(0,s^{-1}P^\bot m,0)$ is a homotopy between $f_{D,1}\pi$ and $\id_{s^{-1}\cone_{j,i}}$ (recall~\eqref{PDP=PD}). This proves the following Theorem.
\begin{theorem}\label{th:modelfiberproduct} The $L_\infty[1]$ algebra $(s^{-1}N\times A,R_D)$, where $R_D$ is defined by equations \eqref{eq:modelfibredproduct1}-\eqref{eq:modelfibredproduct3}, is a weak model for the homotopy fiber product $N\times^h_M L$ along $j:(N,D,[\cdot,\cdot])\rh(M,D,[\cdot,\cdot])$ and $i:(L,D,[\cdot,\cdot])\rh(M,D,[\cdot,\cdot])$.\end{theorem}
\begin{remark}\label{rem-earlierversions} In fact, with the previous notations, it can be proved that $R_D$ is the $L_\infty[1]$ structure induced from $Q_D$ via homotopy transfer along the homotopy retraction data $f_{D,1}$, $\pi$, $K$: this can be done by adapting the computations in Proposition \ref{prop.HDB}.\end{remark}
We now proceed to the proof of Theorem \ref{ThoeremHDBvshomotopyfiber}.
\begin{proof} (of Theorem \ref{ThoeremHDBvshomotopyfiber}) The fact that $\Phi(D)$ is a $L_\infty[1]$ structure is an obvious consequence of Theorem~\ref{TheoremHDB}, the fact that $(A,\Phi(D))$ is a weak model for the homotopy fiber $K_i$ is the particular case of the previous theorem when $N=0$. In fact in \eqref{eq:F_D1}-\eqref{eq:F_D3} we constructed an explicit $L_\infty[1]$ weak equivalence $F_D:A\rh s^{-1}\cone_i$ between $(A,\Phi(D))$ and Fiorenza-Manetti's $L_\infty[1]$ mapping cocone of $i$. The composition of $F_D$ and the strict $L_\infty[1]$ morphism $p_{s^{-1}L}:s^{-1}\cone_i\rh\Sigma^{-1}L$ is given in Taylor coefficients as in equations \eqref{morfismo}. Finally we have a commutative diagram
\[\xymatrix{A\ar[r]\ar[d]^{F_D}&\Sigma^{-1}L\ar[r]^-{\Sigma^{-1}(i)}\ar@{=}[d]&\Sigma^{-1}M\ar@{=}[d]\\
s^{-1}\cone_i\ar[r]^-{p_{s^{-1}L}}&\Sigma^{-1}L\ar[r]^-{\Sigma^{-1}(i)}&\Sigma^{-1}M}\]
By the results of Fiorenza and Manetti \cite{FMcone} the bottom sequence is a homotopy fiber sequence, in fact a model for $\Sigma^{-1}K_f\xrightarrow{\Sigma^{-1}(p_L)}\Sigma^{-1}L\xrightarrow{\Sigma^{-1}(f)}\Sigma^{-1}M$, and so we are done.
\end{proof}
\begin{remark} As for the diagram in the introduction, after d\'ecalage the $L_\infty[1]$ structure on $s^{-1}M\times A$ is $R_D$ in Theorem \ref{th:modelfiberproduct}. The vertical pair of quasi inverses $L_\infty[1]$ quasi isomorphisms are the composition $(s^{-1}M\times A,R_D)\xrightarrow{F_D}(s^{-1}\cylinder_i,Q_D)\xrightarrow{p_{s^{-1}L}}\Sigma^{-1}(L,D,[\cdot,\cdot])$, $F_D$ as in \eqref{eq:F_D1}-\eqref{eq:F_D3}, and the strict $L_\infty[1]$ morphism $\Sigma^{-1}(L,D,[\cdot,\cdot])\rh (s^{-1}M\times A,R_D):s^{-1}l\rh(s^{-1}l,0)$.
\end{remark}
\begin{remark}\label{rem.generalizedbrackets} In this remark we address the problem of how to generalize Theorem \ref{ThoeremHDBvshomotopyfiber} and Theorem~\ref{TheoremHDB} when we remove the assumption that $A\subset M$ is a graded Lie subalgebra. We sketch a proof that in this case one can still construct a correspondence $\Phi:\Der(M/L)\rtimes M\rh\Coder(SA)$ such that both theorems hold. To construct $\Phi$ it is sufficient to construct the corresponding $L_\infty[1]$ extension and this can be done using homotopy transfer, following the proof of Proposition \ref{prop.HDB} up to Remark \ref{rem-nonLie}. We obtain a $L_\infty[1]$ structure $R$ on $s^{-1}\Der(M/L)\times s^{-1}M\times A$: $r_1$ is as in Proposition~\ref{prop.HDB}, while a direct computation shows that $r_2$ is almost as in Proposition \ref{prop.HDB}, except that one should put $\Phi(m)_1(a)=P[m,a]-\frac{1}{2}P[Pm,a]$. An analysis of the several possibilities implies that $r_{i+1}$, for $i\geq2$,  vanishes everywhere but on mixed terms of type $s^{-1}D\otimes a_1\odot\cdots\odot a_i$ and $s^{-1}m\otimes a_1\odot\cdots\odot a_i$: moreover, $A\subset s^{-1}\Der(M/L)\times s^{-1}M\times A$ is a $L_\infty[1]$ ideal. Finally, putting all these things together, we see that we have in fact constructed a $L_\infty[1]$ extension of base $\Sigma^{-1}(\Der(M/L)\rtimes M,0,[\cdot,\cdot])$ and fibre $(A,0)$, classified by a \emph{strict} $L_{\infty}[1]$ morphism which is the desired $\Sigma^{-1}(\Phi)$ (it is also easy to check that the restriction of $\Phi$ to $\Der(M/L)$ factors through the inclusion $\Coder(\overline{SA})\rh\Coder(SA)$). Explicit formulas for $\Phi$ will have to be more involved than those in Definition~\ref{def.HDB}: for instance one can notice that if $A$ is not $[\cdot,\cdot]$-closed, then there is no guarantee that it will be closed with respect to the brackets from Definition \ref{def.HDB}, alternatively one could try to compute the first brackets directly along the previous lines. The proof of Theorem \ref{th:modelfiberproduct} can be repeated verbatim, except that the explicit formulas for $F_D$ do not longer hold (the formula for $f_{D,1}$, however, does), thus by defining the morphism $A\rh\Sigma^{-1} L$ as the composition of $F_D$ and $p_{s^{-1}L}$ (this is no longer given by Equation \eqref{morfismo}), also the proof of Theorem \ref{ThoeremHDBvshomotopyfiber} can be repeated verbatim.
\end{remark}

We close this section with two observations. The first one is an immediate corollary of Theorem~\ref{ThoeremHDBvshomotopyfiber} and Theorem \ref{th:Manetti}. The second one should be confronted with the results of \cite{CS08}.

\begin{corollary}\label{cor.homabelian} In the hypotheses of Theorem \ref{ThoeremHDBvshomotopyfiber}, if $H(i):H(L,D)\rh H(M,D)$ is injective (equivalently, if $H(P):H(M,D)\rh H(A,PD)$ is surjective), then the $L_{\infty}[1]$ algebra $(A,\Phi(D))$ is homotopy abelian.\end{corollary}

\begin{proposition}\label{corollaryCS} Let $i:(L,D,[\cdot,\cdot])\rh(M,D,[\cdot,\cdot])$ be the inclusion of a sub dgla and let $A_k$, $k=1,2$, be a complement to $L$ in $M$, with the $L_\infty[1]$ structure $\Phi_k(D)$ from Theorem \ref{ThoeremHDBvshomotopyfiber} (or from Remark \ref{rem.generalizedbrackets} if we do not wish to assume $A_k\subset M$ a graded Lie subalgebra). Then $(A_1,\Phi_1(D))$ and $(A_2,\Phi_2(D))$ are isomorphic $L_\infty[1]$ algebras.\end{proposition}
\begin{proof} Let $P_k:M\rh A_k$ be the projection with kernel $L$. Let $(s^{-1}\cone_i, Q_D)$ be the desuspended mapping cocone of the inclusion $i$: there is a $L_\infty[1]$ morphism $(A_1,\Phi_1(D))\rh (s^{-1}\cone_i,Q_D)$, with linear Taylor coefficient $a\rh (s^{-1}P_1^\bot Da, a)$, and a $L_\infty[1]$ morphism $(s^{-1}\cone_i,Q_D)\rh(A_2,\Phi_2(D))$, with linear Taylor coefficient $(s^{-1}l,m)\rh P_2m$ (cf. Remark \ref{rem-inversetoF} and Remark \ref{rem-earlierversions}). The composite $(A_1,\Phi_1(D))\rh~(A_2,\Phi_2(D))$ has linear Taylor coefficient $a\rh~P_2a$, which is a dg isomorphism, thus it is a $L_\infty[1]$ isomorphism. \end{proof}

\section{Examples and applications}\label{sec-examples}

\begin{example}\label{ex-chuanglazarev} Recall (cf. \cite{voronov,FrZam2}) that every $L_\infty[1]$ structure can be obtained via higher derived brackets. Let $V$ be a graded space, and identify it with an abelian Lie subalgebra of $\Coder(SV)$ via the section $v\rh \sigma_v$ from Remark \ref{remark2.1}, splitting the exact sequence from Remark \ref{rem.exsequence1}. Evaluation at 1 is then identified with a projection $P:\Coder(SV)\rh V:R\rh R(1)=r_0(1)$, whose kernel is $\Coder(\overline{SV})$: thus we are in the hypotheses of Voronov's construction of higher derived brackets. We claim that the morphism $\Phi:\Coder(SV)\rh\Coder(SV):R\rh\Phi(R)$ of graded Lie algebras, as in Theorem \ref{TheoremHDB}, coincides with the identity. In fact we see from Remark \ref{remark2.1} that $\Phi(R)_0(1)=PR=r_0(1)$ and $\Phi(R)_n(v_1\odot\cdots\odot v_n)=P[\cdots[R,\sigma_{v_1}]\cdots,\sigma_{v_n}]=r_n(v_1\odot\cdots\odot v_n)$, that is, $\Phi(R)=R$ as claimed. If $Q$ is a $L_\infty[1]$ structure on $V$, by the above $Q=\Phi(Q)=\Phi([Q,\cdot])$: then Theorem \ref{ThoeremHDBvshomotopyfiber} implies the following result, already proved in \cite{ChLaz}.
\begin{theorem} A $L_\infty[1]$ algebra $(V,Q)$ is weakly equivalent to $\Sigma^{-1}K_i$, where $K_i$ is the homotopy fiber of the inclusion of dglas
\[ i:(\Coder(\overline{SV}),[Q,\cdot],[\cdot,\cdot])\xrightarrow{}(\Coder(SV),[Q,\cdot],[\cdot,\cdot] ) \]\end{theorem}
\begin{remark} The $L_\infty[1]$ morphism $\operatorname{Ad}:(V,Q)\rh\Sigma^{-1}(\Coder(\overline{SV}),[Q,\cdot],[\cdot,\cdot])$ from Theorem \ref{ThoeremHDBvshomotopyfiber} is given in Taylor coefficients $p\operatorname{Ad}=(\operatorname{Ad}_1,\ldots,\operatorname{Ad}_i,\ldots)$ by
\[ \operatorname{Ad}_i:V^{\odot i}\rh s^{-1}\Coder(\overline{SV}): v_1\odot\cdots\odot v_i\rh s^{-1}([\cdots[Q,\sigma_{v_1}]\cdots,\sigma_{v_i}]-\sigma_{q_{i}(v_1\odot\cdots\odot v_i)})\]
(more explicitly $s\operatorname{Ad}_i(v_1\odot\cdots\odot v_i)_k(v_{i+1}\odot\cdots\odot v_{i+k})=q_{i+k}(v_1\odot\cdots\odot v_{i+k})$ for every $k\geq1$, where we denote by $s: s^{-1}\Coder(\overline{SA})\rh\Coder(\overline{SA})$ the shift map). This is the $L_\infty[1]$ generalization of the adjoint morphism of a dgla introduced by Chuang and Lazarev in~\cite{ChLaz}. \end{remark}
As a consequence of Corollary \ref{cor.homabelian} and some elementary homological algebra, we obtain the following proposition, which could also be deduced by the results in \cite{ChLaz}.
\begin{proposition}\label{homabcrit} Let $(V,Q)$ be a $L_\infty[1]$ algebra. If the evaluation morphism
\[ e:(\Coder(SV),[Q,\cdot])\rh (V,q_1):R\rh R(1)= r_0(1)\]
admits a dg right inverse, then the $L_\infty[1]$ algebra $(V,Q)$ is homotopy abelian.\end{proposition}
\begin{remark} The converse of the above proposition is also true, so the hypothesis is a necessary and sufficient condition for a $L_\infty[1]$ algebra $(V,Q)$ to be homotopy abelian, cf. Theorem 3.6 in \cite{FoKBopLA}. We notice that among the other equivalent conditions stated there, one could add the vanishing of the map induced by the adjoint $H(\operatorname{Ad}_1):H(V,q_1)\rh H_{\overline{CE}}(V,V)[1]$ from the tangent cohomology to the reduced Chevalley-Eilenberg cohomology of $(V,Q)$ with coefficients in itself, where the $L_\infty[1]$ adjoint morphism $\operatorname{Ad}$ is defined as in the previous remark.\end{remark}
\end{example}

\begin{example}\label{ex.Poisson} Higher derived brackets have been applied in the study of coisotropic deformations, cf. \cite{Sch,CS08,FrZam1}. Let $X$ be a differentiable manifold, and let $TX$ be the tangent bundle: we denote by $\sV^{\ast}_X=\Gamma(\bigwedge^{\ast} TX)$ the Gerstenhaber algebra of polyvector fields on $X$, equipped with the Schouten-Nijenhuis bracket $[\cdot,\cdot]_{SN}:\sV^{i}_X\ten\sV^j_X\rh\sV^{i+j-1}_X$. Recall that a Poisson structure on $X$ is the datum of a bivector $\pi\in\sV^2_X$ such that $[\pi,\pi]_{SN}=0$: as well known (cf. for instance \cite{BM}, Section 3) this induces a Poisson bracket $\{\cdot,\cdot\}_{\pi}$ on the algebra $A(X)$ of smooth functions on $X$, a dgla structure $(s^{-1}\sV^{\ast}_X,[\pi,\cdot]_{SN},[\cdot,\cdot]_{SN})$ on the desuspension $s^{-1}\sV^{\ast}_X$, and finally an anchor map $\pi^{\#}:T^{\ast}X\rh TX$ given by contraction with $\pi$.

Let $Z\subset X$ be a closed smooth submanifold, recall that $Z$ is coisotropic if the vanishing ideal $I(Z)\subset A(X)$ is $\{\cdot,\cdot\}_{\pi}$-closed, equivalently, if $\pi^{\#}(N^{\ast}Z)\subset TZ$, where $N^{\ast}Z\subset T^{\ast}X$ is the annihilator of $TZ$. Let $NZ$ be the normal bundle of $Z$ in $X$, $\sN^{\ast}_{Z|X}=\Gamma(\bigwedge^{\ast} NZ)$. Restriction to $Z$ followed by projection induces an algebra epimorphism $\sV^{\ast}_X\rh\sN^{\ast}_{Z|X}$: let $\sL^{\ast}_Z\subset\sV^{\ast}_X$ be defined by the exact sequence
\begin{equation}\label{exsequencepoisson} 0\rh\sL^{\ast}_Z\rh\sV^{\ast}_X\rh\sN^{\ast}_{Z|X}\rh0. \end{equation}
As in \cite{BM}, Proposition 5.2, $\sL^{\ast}_Z$ is a Gerstenhaber subalgebra of $\sV^{\ast}_X$, and $Z$ is coisotropic if and only if $\pi\in\sL^2_Z$. When $X$ is the total space of a vector bundle on $Z$ (and $Z$ is embedded as the zero section) the above sequence admits a natural splitting, whose desuspension sends $s^{-1}\sN^{\ast}_{Z|X}$ onto the abelian Lie subalgebra of $s^{-1}\sV^{\ast}_X$ consisting of vertical polyvector fields constant along the fibers, cf. for instance \cite{CS08,Sch}: in general one reduces to this situation via the choice of an embedding of $NZ$ onto a tubular neighborhood of $Z$ in $X$. For a coisotropic $Z$ the higher derived brackets $\Phi(\pi)=\Phi([\pi,\cdot]_{SN})$ induce a $L_{\infty}[1]$ structure on $s^{-1}\sN^{\ast}_{Z|X}$: moreover, it follows from Proposition \ref{corollaryCS} that the resulting $L_\infty[1]$ algebra is independent from the involved choices up to isomorphism, obtaining a result already proved in \cite{CS08}.

It is known \cite{Sch} that the $L_\infty[1]$ algebra $(s^{-1}\sN^{\ast}_{Z|X},\Phi(\pi))$ governs the functor of infinitesimal coisotropic deformations of $Z$ in $X$ (via the associated deformation functor, cf. \cite{Man,FMcone,BM}): as weakly equivalent $L_\infty[1]$ algebras determine the same deformation functor, Theorem \ref{ThoeremHDBvshomotopyfiber} implies the following
\begin{theorem} Let $(X,\pi)$ be a differentiable Poisson manifold, $Z\subset X$ a coisotropic submanifold, then the homotopy fiber of the inclusion of dglas
\[ i:(s^{-1}\sL^{\ast}_Z,[\pi,\cdot]_{SN},[\cdot,\cdot]_{SN})\rh(s^{-1}\sV^{\ast}_X,[\pi,\cdot]_{SN},[\cdot,\cdot]_{SN}) \]
governs the functor of infinitesimal coisotropic deformations of $Z$ in $X$.\end{theorem}
The same result could have been proved by the methods of \cite{BM}.
\end{example}

\begin{example} Let $(V,d)$ be a dg space and let $W\subset V$ be a dg subspace, let $A\subset V$ be a complement to $W$ in $V$, thus $V=W\oplus A$ as graded space; we denote by $P:V\rh A$ the projection with kernel $W$ and by $P^\bot:=\id_V-P$. The map $\widetilde{P}:\End(V)\rh\End(V):f\rh PfP^\bot$ is a projection with kernel $\End(V/W)$, the graded Lie subalgebra of $\End(V)$ consisting of those $f:V\rh V$ such that $f(W)\subset W$, and image canonically isomorphic to $\Hom(W,A)$: we notice that $f$ lies in $\operatorname{Im}\,\widetilde{P}$ if and only if $f=Pf=fP^\bot=PfP^\bot$, thus if $f,g\in\operatorname{Im}\,\widetilde{P}$, then $fg=fP^\bot Pg=0$ and we see that $\operatorname{Im}\,\widetilde{P}$ is an abelian Lie subalgebra of $\End(V)$. By hypotheses $d$ is a Maurer Cartan element of $\End(V/W)$, hence it induces via higher derived brackets a $L_\infty[1]$ structure $\Phi(d)=\Phi([d,\cdot])$ on $\Hom(W,A)$. The linear bracket is $\Phi(d)_1(f)(w)=Pdf(w)-(-1)^{|f|}f(dw)$, which is the induced differential on $\Hom((W,d),(A,Pd))$. Next we observe that if $f,g\in\operatorname{Im}\,\widetilde{P}$ then
\[[[d,f],g]=(-1)^{|f|+1}\left(fdg-(-1)^{(|f|+1)(|g|+1)}gdf\right)=(-1)^{|f|+1}P\left(fdg-(-1)^{(|f|+1)(|g|+1)}gdf\right)P^\bot\]
thus $[[d,f],g]\in \operatorname{Im}\,\widetilde{P}$, which implies $\Phi(d)_i=0$ for $i\geq3$. Finally the above computation shows
\[\Phi(d)_2(f\odot g)(w)=(-1)^{|f|+1}\left(f(P^\bot dg(w))-(-1)^{(|f|+1)(|g|+1)}g(P^\bot df(w))\right)\]
Via d\'ecalage, it is defined a dgla structure on $\Hom(W,A)[-1]$, weakly equivalent to the homotopy fiber of the inclusion of dglas $(\End(V/W),[d,\cdot],[\cdot,\cdot])\rh(\End(V),[d,\cdot],[\cdot,\cdot])$: by homotopy invariance and results of Fiorenza and Manetti (\cite{FMperiod}, cf. also the related \cite{FMperiod2}), the associated deformation functor is the dg Grassmann functor controlling the infinitesimal embedded deformations of the subcomplex $(W,d)$ in $(V,d)$ (modulo an opportune Gauge equivalence relation, cf. \cite{FMperiod} for details). Finally, in this case the $L_\infty$ morphism \eqref{morfismo} is a strict morphism of dglas
\[\Hom(W,A)[-1]\rh \End(V/W): f\rh [d,f]-P[d,f]P^\bot=[P^\bot dP,f].\]

\end{example}

\begin{example}\label{ex-koszulbrackets} Let $A$ be a unital associative graded algebra, we denote by $A_{L}$ the corresponding graded Lie algebra with the commutator bracket, and by $A_J$ the corresponding graded Jordan algebra with Jordan product $a\circ b=\frac{1}{2}(ab+(-1)^{|a||b|}ba)$. We identify $A_L$ with a Lie subalgebra of $\End(A)$ via the embedding $\bl:A_{L}\rh\End(A):a\rh\bl_a$, where $\bl_a$ is the operator of left multiplication by $a$. There is a projection $P:\End(A)\rh A_L:f\rh\bl_{f(1)}$, whose kernel $L=\{f\in\End(A) \mbox{\,\,s.t.\,\,} f(1) = 0 \}$ is a graded Lie subalgebra of $\End(A)$. We are in the set up of Section \ref{sectionHDB}, so higher derived brackets define a morphism of graded Lie algebras $\Phi:\End(A)\rh\Coder(SA):f\rh\Phi(f)$. When $A$ is graded commutative the brackets defined in this way are the usual higher Koszul brackets $\sK(f)_i$ on $A$ associated to $f$ \cite{koszul,voronov}, in the non commutative case we expect to recover the hierarchy  of nonabelian higher Koszul brackets introduced by Bering in \cite{Bering}, Section 3.

To get a taste of how do the just defined brackets look like in the non commutative case, one can easily verify that $\Phi(f)_1(a)=f(a)-\frac{1}{2}(f(1)a+(-1)^{|a||f|}af(1))=f(a)-f(1)\circ a$ for every $a\in A$; as another example suppose $f\in\End(A)$ is such that $f(1)=0$, then it can be checked that $\Phi(f)_2(a\odot b)=f(a\circ b)-f(a)\circ b - (-1)^{|a||f|}a\circ f(b)$.
\end{example}

\begin{remark}\label{rem-differentialoperators} We could define a family of subspaces $D_k=\{ f\in\End(A) \mbox{\,\,s.t.\,\,}\Phi(f)_{i}=0,\,\, \forall\,i>k\}$ for $k\geq0$: since $\Phi$ is a morphism of Lie algebras the identity $[D_i,D_j]\subset D_{i+j-1}$ follows immediately, in particular $\bigcup_{k\geq0} D_k$ is a graded Lie subalgebra of $\End(A)$. When $A$ is graded commutative $D_k=\Diff_k(A)\subset\End(A)$ is the subspace of differential operators of order $\leq k$ on $A$: it is not clear to the author whether the $D_k$ define interesting classes of operators in the non commutative case as well, perhaps a more meaningful construction could be made along the lines of the next remark.\end{remark}

\begin{remark} It should be possible to construct higher derived operations in different operadic contexts using homotopy transfer along the lines of this paper. For instance, let $B$ be a graded associative algebra with a splitting, as a graded space, in the direct sum $B=A\oplus C$ of graded subalgebras $A$ and $C$, then for every $D\in\Der^1(B/A)$ such that $D^2=0$ it is induced, via homotopical transfer of structure from the homotopy fiber of $i:(A,D,\cdot)\rh (B,D,\cdot)$, an $A_\infty[1]$ structure on $C$, that is, a squaring to zero degree one coderivation on the reduced tensor coalgebra $\overline{TC}$ over $C$: then we could define the higher derived products $\Phi(D)_i:C^{\otimes i}\rh C$ on $C$ associated to $D$ as the Taylor coefficients of this structure. A case of interest should be when $A\subset B$ is a left ideal (this is the situation considered in \cite{Ku}), for instance in the setting of the previous example.\end{remark}

\begin{example}\label{ex-bvinfinity} The author is grateful to M. Manetti for pointing out and carefully explaining to him the following example.
\begin{definition} Let $k$ be a fixed odd integer, a commutative $BV_\infty$ algebra $(A,\Delta_0,\Delta_1,\ldots,\Delta_i,\ldots)$ of degree $k$ consists of the following data (\cite{Krav}, cf. also \cite{BrLaz}):
\begin{itemize}
\item[(1)] a commutative dga $(A,\Delta_0)$ with 1, and
\item[(2)] for every $i\geq1$ a differential operator $\Delta_i\in\Diff^{1-n(k+1)}_{i+1}(A)$ on $A$ of degree $1-n(k+1)$ and order $\leq i+1$ (cf. remark \ref{rem-differentialoperators}), such that
\item[(3)]  $\Delta_i(1)=0$ for every $i\geq0$, and
\item[(4)]   if we denote by $t$ a central variable of (even) degree $k+1$ then the degree one operator
\[ \Delta = \Delta_0 + t\Delta_1 + \cdots + t^i\Delta_i + \cdots \]
on the algebra of formal power series $A[[t]]$ (that is $\Delta(\sum_j a_j\cdot t^j):= \sum_{i,j} \Delta_i(a_j)\cdot t^{i+j}$) squares to zero.
\end{itemize}\end{definition}
There is an associated $L_\infty[1]$ structure on $A[k+1]$ as we now describe. Consider the algebra of formal Laurent series $A((t))=\bigcup_{j\in\Z} t^{j}A[[t]]$, we denote by $p_+:A((t))\rh A[[t]]$ the natural projection, by $A((t))^-$ its kernel and by $p_-=\id_{A((t))}-p_+:A((t))\rh A((t))^-$ . The graded Lie algebra $M=\End( A((t)) )$ splits as $M=L\oplus B$, where $L=\{f\in M\mbox{ s.t. } f(1)\in A[[t]]\}$ and $B\subset M$ is the abelian Lie subalgebra of operators of left multiplication by elements in $A((t))^-$. Obviously $L\subset M$ is a Lie subalgebra, moreover, $\Delta:A[[t]]\rh A[[t]]$ extends by $\K((t))$-linearity to $\Delta:A((t))\rh A((t))$, which is a Maurer-Cartan element in $L$: we are in the set up of Voronov's construction of higher derived brackets, so $\Phi(\Delta)$ defines a $L_\infty[1]$ structure on $B$.
\begin{proposition} Let $\bi:A[k+1]\rh B $ be the linear embedding sending $a$ to the operator of left multiplication by $a\cdot t^{-1}$, then $\bi(A[k+1])$ is a $L_\infty[1]$ subalgebra of $(B,\Phi(\Delta))$. The induced $L_\infty[1]$ algebra structure on $A[k+1]$ is $(A[k+1],\Delta_0,\sK(\Delta_1)_2,\ldots,\sK(\Delta_{i-1})_i,\ldots)$, where $\sK(f)_i$ denotes the $i$-th Koszul bracket on $A$ associated to $f\in\End(A)$ (cf. Example \ref{ex-koszulbrackets}).
\end{proposition}
\begin{proof} A straightforward computation shows
\[ [\cdots[\Delta,\bi(a_1)]\cdots,\bi(a_i)](1) = [ [\Delta_0 + t\Delta_1 + \cdots + t^j\Delta_j + \cdots,\,\,\bl_{a_1\cdot t^{-1}} ]\cdots,\,\, \bl_{a_i\cdot t^{-1}}](1) = \]
\[ = \sum_{j\geq0} [\cdots[\Delta_j,\bl_{a_1}]\cdots,\bl_{a_i}](1)\cdot t^{j-i} = \sum_{j\geq 0}\sK(\Delta_j)_i(a_1\odot\cdots\odot a_i)\cdot t^{j-i}\]
by definition of the Koszul brackets, thus $\Phi(\Delta)_i(\bi(a_1)\odot\cdots\odot\bi(a_i))$ is left multiplication with $\sum_{j=0}^{i-1}\sK(\Delta_j)_i(a_1\odot\cdots\odot a_i)\cdot t^{j-i}$. Assumption (2) in the definition of a commutative $BV_\infty$ algebra implies that $\sK(\Delta_j)_i=0$ if $j<i-1$, so we see that
\[\Phi(\Delta)_i(\bi(a_1)\odot\cdots\odot\bi(a_i))=\bi(\sK(\Delta_{i-1})_i(a_1\odot\cdots\odot a_i))\]. \end{proof}
\begin{definition} A commutative $BV_\infty$ algebra has the degeneration property if the projection $(A[[t]],\Delta)\rh (A,\Delta_0):a(t)\rh a(0)$ is surjective in homology. \end{definition}
The following theorem, which was proved with different methods by Braun and Lazarev in \cite{BrLaz}, generalizes the formality theorem from \cite{Lieformality}. We follow the proof of this last result given in \cite{iacono}, Theorem 6.6.
\begin{theorem} If a commutative $BV_\infty$ algebra $(A,\Delta_0,\Delta_1,\ldots,\Delta_i,\ldots)$ of (odd) degree $k$ has the degeneration property, then the $L_\infty[1]$ algebra $(A[k+1],\Delta_0,\sK(\Delta_1)_2,\ldots,\sK(\Delta_{i-1})_i,\ldots)$ is homotopy abelian.
\end{theorem}
\begin{proof} Consider the decreasing filtration $A((t))=\bigcup_{j\in\Z} F^j = \bigcup_{j\in\Z}t^j A[[t]]$ by $\Delta$-closed subspaces: then the degeneration property is equivalent to injectivity in homology of $F^1\subset F^0$, which readily implies injectivity in homology of $F^j\subset F^{j-1}$ (by $\K((t))$-linearity of $\Delta$) and also of $F^j\subset F^k$ for every $j>k$. In particular $A[[t]]\rh A((t))$ is injective in homology, so $p_-:A((t))\rh A((t))^-$ is surjective in homology: in turn this also implies that the projection $P:M\rh B$ induced by the splitting $M=L\oplus B$ is surjective in homology, this follows by looking at the commutative diagram of dg spaces
\[\xymatrix{(M,[\Delta,\cdot])\ar[r]^P\ar[d]^{\operatorname{ev}_1} & (B,P[\Delta,\cdot])\ar[d]^{\operatorname{ev}_1}\\
(A((t)),\Delta)\ar[r]^-{p_-}& (A((t))^-,p_-\Delta)}\]
Since we are working over a field $H(\operatorname{ev}_1)=\operatorname{ev}_{[1]}:H(M)=\End(H(A((t))))\rh H(A((t)))$ is surjective, while $\operatorname{ev}_1:B\rh A((t))^-$ is an isomorphism. Thus the $L_\infty[1]$ structure $\Phi(\Delta)=\Phi([\Delta, \cdot])$ on $B$ is homotopy abelian by Corollary \ref{cor.homabelian}.

The thesis follows from \cite{KKP}, Proposition 4.11, if we show that the embedding $\bi:A[k+1]\rh B$ from the previous proposition is injective in homology. We look at the commutative diagram
\[ \xymatrix{0\ar[r]&F^0\ar@{=}[d]\ar[r] & F^{-1}\ar[d]\ar[r] & F^{-1}/F^0\ar[d]\ar[r]&0 \\0\ar[r]& A[[t]]\ar[r] & A((t))\ar[r] & A((t))^-\ar[r]&0 } \]
The rows are split exact and the middle vertical arrow is injective in homology, then so must be the right one: but this is isomorphic to $\bi:A[k+1]\rh B$.\end{proof}
\end{example}

\begin{example}\label{ex-getzler} Let $(L,D,[\cdot,\cdot])$ be a dgla, then $L=L^{\geq0}\oplus L^{<0}$, where $L^{\geq0}=\oplus_{i\geq0}L^i$ and $L^{<0}=\oplus_{i<0}L^i$: we notice that both $L^{\geq0}$ and $L^{<0}$ are Lie subalgebras of $L$ and $D\in\Der(L/L^{\geq0})$, thus we are in the set up of Section \ref{sectionHDB} and the higher derived bracket associated to $D$ induce a $L_\infty[1]$ structure $\Phi(D)$ on $L^{<0}$. The brackets are explicitly given, for $i\geq2$, by
\[ \Phi(D)_i(l_1\odot\cdots\odot l_i)=\sum_{\sigma\in S_i}\varepsilon(\sigma)\sum_{k=1}^{i}\frac{B_{i-k}}{k!(i-k)!}\overbrace{[\cdots[}^{i-k}P([\cdots[ Dl_{\sigma(1)},l_{\sigma(2)}]\cdots]),l_{\sigma(k+1)}]\cdots,l_{\sigma(i)}] =\]
\[ = \sum_{\sigma\in S_i}\varepsilon(\sigma)\left(\frac{B_{i-1}}{(i-1)!}[\cdots[PDl_{\sigma(1)},l_{\sigma(2)}]\cdots,l_{\sigma(i)}]+
\left(\sum_{k=2}^i\frac{B_{i-k}}{k!(i-k)!}\right)[\cdots[Dl_{\sigma(1)},l_{\sigma(2)}]\cdots,l_{\sigma(i)}]
\right) =\]
\[ = -\frac{B_{i-1}}{(i-1)!}\sum_{\sigma\in S_i}\varepsilon(\sigma)[\cdots[P^\bot Dl_{\sigma(1)},l_{\sigma(2)}]\cdots,l_{\sigma(i)}]\] 
where in the second identity we used the fact that for $k>1$ the $P$ becomes irrelevant, since it applies to an element already in $L^{<0}$, and in the third one we used the identity $\sum_{k=0}^{i-1}B_k\left(\begin{array}{c} i \\ k \end{array}\right)=0$ for $i\geq2$. We notice that $P^\bot D$ acts on $L^{<0}$ as $D$ on $L^{-1}$ and 0 elsewhere. The linear bracket is $\Phi(D)_1=PD$, which is $0$ on $L^{-1}$ and $D$ on $L^{<-1}$. These are essentially the same brackets as those introduced by Getzler in \cite{Getzler}. Finally, by Theorem \ref{ThoeremHDBvshomotopyfiber} we see that the strict $L_\infty[1]$ morphism  $L^{<0}\rh\Sigma^{-1}L^{\geq0}:l\rh s^{-1}P^\bot Dl $ fits into a homotopy fiber sequence $L^{<0}\rh\Sigma^{-1} L^{\geq0}\rh\Sigma^{-1}L$ of $L_\infty[1]$ algebras.
\end{example}


\begin{thebibliography}{99}

\bibitem{FoKBopLA} R. Bandiera, \emph{Formality of Kapranov's brackets on pre-Lie algebras}; \texttt{arXiv:1307.8066 [math.QA]}.

\bibitem{BM} R. Bandiera, M. Manetti, \emph{On coisotropic deformations of holomorphic submanifolds}; \texttt{arXiv:1301.6000v2 [math.AG]}.

\bibitem{berglund} A. Berglund, \emph{Homological perturbation theory for algebras over operads}; \texttt{arXiv:0909.3485v2 [math.AT]}.

\bibitem{Bering} K. Bering, \emph{Non-commutative Batalin-Vilkovisky algebras, homotopy Lie algebras and the Courant bracket}, Comm. Mat. Phys. \textbf{274} (2007), 297-341; \texttt{arXiv:hep-th/0603116}.

\bibitem{BrLaz} C. Braun, A. Lazarev, \emph{Homotopy BV algebras in Poisson geometry}; \texttt{arXiv:1304.6373 [math.QA]}.

\bibitem{CS08} A. Cattaneo, F. Sch\"{a}tz,
\emph{Equivalences of higher derived brackets}
J. Pure Appl. Algebra \textbf{212} (2008), no. 11, 2450-2460;
\texttt{	arXiv:0704.1403v3 [math.QA]}.

\bibitem{ChLaz}
J.~ Chuang, A.~ Lazarev, \emph{L-infinity maps and twistings},  Homology Homotopy Appl. \textbf{13}, no. 2 (2011), 175-195;
\texttt{arXiv:1108.1525}.


\bibitem{ChLaz2}
J.~ Chuang, A.~ Lazarev, \emph{Combinatorics and formal geometry of the master equation},  Letters in Math. Phys. \textbf{103}, no. 1 (2013), 79-112;
\texttt{arXiv:1205.5970}.

\bibitem{FMcone} D. Fiorenza, M. Manetti,
\emph{$L_{\infty}$ structures on mapping cones}, Algebra \& Number Theory \textbf{1} (2007), 301-330;
\texttt{arXiv:0601312 [math.QA]}.


\bibitem{FMperiod} D. Fiorenza, M. Manetti,
\emph{L-infinity algebras, Cartan homotopies and period maps};
\texttt{arXiv:math/06055297 [math.AG]}.

\bibitem{FMperiod2} D. Fiorenza, M. Manetti,
\emph{A period map for generalized deformations}, Journal of Noncomm. Geom. \textbf{3}, no.4 (2009), 579-597;
\texttt{arXiv:0808.0140 [math.AG]}.

\bibitem{FrZam1} Y. Fr\'egier, M. Zambon,
\emph{Simultaneous deformations and Poisson geometry};
\texttt{arXiv:1202.2896 [math.QA]}.

\bibitem{FrZam2} Y. Fr\'egier, M. Zambon,
\emph{Simultaneous deformations of algebras and morphisms via derived brackets};
\texttt{arXiv:1202.2896v1 [math.QA]}.


\bibitem{Getzler04} E.~Getzler, \emph{Lie theory for nilpotent $L_{\infty}$-algebras.}, Ann. of Math. \textbf{170}, no. 1 (2009), 271-301; \texttt{arXiv:math/0404003v4}.


\bibitem{Getzler} E. Getzler, \emph{Higher derived brackets}; \texttt{arXiv:1010.5859v2 [math-ph]}.

\bibitem{hinichdgC}
V. Hinich, \emph{DG coalgebras as formal stack},
J. Pure Appl. Algebra \textbf{162} (2001), 209-250;
\texttt{arXiv:math/9812034v1 [math.AG]}.



\bibitem{iacono0} D. Iacono, \emph{$L_\infty$-algebras and deformations of holomorphic maps}, Internat. Math. Res. Notices \textbf{8}, (2008); \texttt{arXiv:0705.4532v2 [math.AG]}.

\bibitem{iacono} D. Iacono, \emph{Deformations and obstructions of pairs $(X,D)$}; \texttt{arXiv:1302.1149 [math.AG]}.


\bibitem{iacMan1} D. Iacono, M. Manetti, \emph{An algebraic proof of Bogomolov-Tian-Todorov theorem}, Deformation Spaces \textbf{39} (2010), 113-133;
\texttt{arXiv:0902.0732v2 [math.AG]}.
\bibitem{iacMan2} D. Iacono, M. Manetti, \emph{Semiregularity and obstructions of complete intersections}, Adv. in Math. \textbf{235} (2013), 92-125; \texttt{arXiv:1112.0425v4 [math.AG]}.

\bibitem{KKP} L. Katzarkov, M. Kontsevich, T. Pantev, \emph{Hodge theoretic aspects of mirror symmetry}, in \emph{From Hodge theory to integrability and TQFT: tt*-geometry}, Proc. Sympos. Pure Math. \textbf{78}, Amer. Math. Soc., Providence (2008), 87-174; \,\,\texttt{arXiv:0806.0107v1 [math.AG]}.

\bibitem{Ku} O. [H.] M. Khudaverdian, \emph{Algebras with operator and Campbell-Haussdorf formula}, Lett. Math. Phys. \textbf{35}, no.~1 (1995), 27-38; \texttt{arXiv:hep-th/9408070}.

\bibitem{KoSo} M. Kontsevich, Y. Soibelman, \emph{Deformation theory I}, draft of the book, available at \texttt{www.math.ksu.edu/~soibel}.

\bibitem{KosmannSch} Y. Kosmann-Schwarzbach, \emph{Derived brackets}, Lett. Math. Phys. \textbf{69} (2004), 61-87; \texttt{arXiv:1212.0559 [math.DG]}.

\bibitem{koszul} J.-L. Koszul,
\emph{Crochet de Schouten-Nijenhuis et cohomologie}, Ast\'erisque, (Numero Hors Serie) (1985) 257-271.

\bibitem{Krav} O. Kravchenko, \emph{Deformations of Batalin-Vilkovisky algebras}, in \emph{Poisson geometry (Warsaw, 1998)}, vol. 51 of
\emph{Banach Center Publ.}, Polish Acad. Sci., Warsaw (2000), 131-139; \texttt{arXiv:9903191 [math.QA]}.

\bibitem{LaSt} T. Lada, J. Stasheff, \emph{Introduction to sh Lie algebras for physicists}, Int. J. Theor. Phys. \textbf{32} (1992), 1087-1104; \texttt{arXiv:hep-th/9209099}.

\bibitem{Lazarev} A.~ Lazarev, \emph{Models for classifying spaces and derived deformation theory}, Proc. London Math. Soc. \textbf{109} (2014), 40-64.; \texttt{arXiv:1209.3866v3}.

\bibitem{Man} M. Manetti, \emph{Lie description of higher obstructions to
deforming submanifolds}, Ann. Sc. Norm. Super. Pisa Cl. Sci.
\textbf{6} (2007) 631-659; {\texttt{arXiv:math.AG/0507287}}.

\bibitem{Merkulov} S. A. Merkulov, \emph{Operad of formal homogeneous spaces and Bernoulli numbers}, Algebra \& Number Theory \textbf{2} (2008), 407-433; \texttt{arXiv:0708.0891 [math.QA]}.

\bibitem{methazambon} R. Mehta, M. Zambon, \emph{$L_\infty$-algebra actions}, Diff. Geom. and its applications \textbf{30} (2012), 576-587;
\texttt{arXiv:1202.2607v2 [math.DG]}.

\bibitem{Sch} F. Sch\"atz, M. Zambon, \emph{Deformations of coisotropic submanifolds for fibrewise entire Poisson structures}; \texttt{arXiv:1207.1696v1 [math.SG]}.

\bibitem{Lieformality} G. Sharygin, D. Talalaev, \emph{On the Lie-formality of Poisson manifolds}, J. K-Theory \textbf{2} (2008), no. 2, Special issue in memory of Yurii Petrovich Solovyev, part 1, 361-384; \texttt{arXiv:math/0503635}.

\bibitem{voronov} Th. Voronov, \emph{Higher derived brackets and homotopy algebras},
J. Pure Appl. Algebra \textbf{202} (2005), 133-153;
\texttt{arXiv:0304038 [math.QA]}.

\bibitem{voronov2} Th. Voronov,
\emph{Higher derived brackets for arbitrary derivations}, Travaux math\'ematiques, fasc. \textbf{XVI},
Univ. Luxemb., Luxembourg (2005), 163-186; \texttt{arXiv:0412202 [math.QA]}.

\bibitem{yeku} A. Yekutieli, \emph{Continuous and twisted $L_\infty$ morphisms}, Journal Pure Appl. Algebra \textbf{207}, no. 3 (2006) , 575-606; \texttt{arXiv:math/0502137v3 [math.QA]}.

\end{thebibliography}
\end{document}